\documentclass[11pt]{article}
\usepackage[utf8]{inputenc}

\usepackage{pgfgantt} 

\usepackage{xcolor}
\definecolor{parchment}{RGB}{214, 204, 169}

\usepackage[margin=1.5in]{geometry}

\usepackage{microtype}
\usepackage{verbatim}
\usepackage{graphicx}
\usepackage{rotating}

\usepackage{svg}

\usepackage{url}
\definecolor{linkColor}{RGB}{156,78,13}
\definecolor{diagramBlue}{RGB}{94, 164 ,255}
\definecolor{diagramGold}{RGB}{255,162,23}
\usepackage{hyperref}
\hypersetup{
    colorlinks = true,
    allcolors = linkColor 
    }
\usepackage{pslatex}

\usepackage{datetime}
\newdateformat{versiondate}{%
\THEMONTH\THEDAY}

\usepackage{pgf,tikz,environ}
\usetikzlibrary{arrows}
\usetikzlibrary{cd}

\usepackage{xpatch} 

\usepackage{soul}

\usepackage{amsthm}
\usepackage{amsmath} 
\usepackage{amsfonts}
\usepackage{amssymb}

\usepackage{sectsty}
\newcommand{\secfont}{\fontfamily{lmss}\selectfont}
\allsectionsfont{\secfont}

\newtheoremstyle{zoltanstyle}
  {1em} 
  {\topsep} 
  {} 
  {} 
  {\bfseries} 
  {.} 
  {.5em} 
  {} 
\renewcommand{\qed}{\unskip\nobreak\quad\newline\textbf{Qed.}}

\theoremstyle{zoltanstyle}
\swapnumbers
\xpatchcmd\swappedhead{~}{.~}{}{}
\newtheorem{body}{}
\numberwithin{body}{section}

\newtheorem{conjecture}[body]{\secfont Conjecture}
\newtheorem{corollary}[body]{\secfont Corollary}
\newtheorem{definition}[body]{\secfont Definition}
\newtheorem{example}[body]{\secfont Example}

\newtheorem{lemma}[body]{\secfont Lemma}

\newtheorem{proposition}[body]{\secfont Proposition}
\newtheorem{remark}[body]{\secfont Remark}

\newtheorem{theorem}[body]{\secfont Theorem}

\usepackage{pgf,tikz,environ}
\usepackage{quiver}
\usetikzlibrary{arrows}
\usetikzlibrary{cd}
\usetikzlibrary{positioning}
\usetikzlibrary{fadings}
\usetikzlibrary{decorations.pathreplacing}
\usetikzlibrary{decorations.pathmorphing}
\tikzset{snake it/.style={decorate, decoration=snake}}

\definecolor{parchment}{RGB}{214, 204, 169}

\expandafter\let\expandafter\oldproof\csname\string\proof\endcsname
\let\oldendproof\endproof

\renewenvironment{proof}[1][\proofname]{%
  \oldproof[\normalfont \bfseries #1.]%
}{\oldendproof}

\usepackage[inline]{enumitem}
\usepackage{mathptmx}

\usepackage[T1]{fontenc}
\usepackage[utf8]{inputenc}

\usetikzlibrary {positioning}
\usetikzlibrary{fadings}
\usetikzlibrary{decorations.pathreplacing}
\usetikzlibrary{decorations.pathmorphing}
\tikzset{snake it/.style={decorate, decoration=snake}}
\definecolor{gold}{RGB}{255,215,0}
\definecolor{softBlack}{RGB}{45, 47, 49}
\definecolor{creamWhite}{RGB}{245,244,241}
\definecolor{softGray}{RGB}{220, 216, 214}
\definecolor{brick}{RGB}{232, 48, 48}

\newcommand{\typesetoperator}[1]{\mathsf{#1}}

\usepackage{xspace}
\definecolor{gold}{RGB}{255,215,0}
\definecolor{softBlack}{RGB}{45, 47, 49}
\definecolor{creamWhite}{RGB}{245,244,241}
\definecolor{softGray}{RGB}{220, 216, 214}
\definecolor{brick}{RGB}{232, 48, 48}

\usepackage{enumitem}

\newcommand{\define}[1]{\textbf{#1}}

\newcommand{\cat}[1]{\mathsf{#1}}

\DeclareMathOperator{\id}{\cat{id}}

\newcommand{\contract}[1]{_{/#1}}

\DeclareMathOperator{\op}{\cat{op}}

\usepackage{adjustbox}

\DeclareMathOperator{\domain}{\typesetoperator{dom}}

\DeclareMathOperator{\mi}{\typesetoperator{mi}}
\DeclareMathOperator{\im}{\typesetoperator{im}}
\DeclareMathOperator{\image}{\typesetoperator{im}}
\DeclareMathOperator{\imageObj}{\typesetoperator{Im}}

\DeclareMathOperator*{\colim}{colim}

\DeclareMathOperator{\Fa}{\mathcal{F}}

\DeclareMathOperator{\La}{\mathcal{L}}  
\DeclareMathOperator{\Ma}{\mathcal{M}}

\DeclareMathOperator{\Qa}{\mathcal{Q}}

\newcommand{\twoloop}{\text{%
\tikz[baseline]{\fill (0,0.5ex) circle (0.15ex);
\draw (-0.3ex,0.1ex) circle (0.5ex);
\draw (0.3ex,0.9ex) circle (0.5ex);
}}
}

\newcommand{\twodiscrete}{\text{%
\tikz[baseline]{\fill (-0.6ex,0.7ex) circle (0.15ex);
\draw (-0.6ex,0.2ex) circle (0.5ex);
\fill (0.6ex,0.3ex) circle (0.15ex);
\draw (0.6ex,0.8ex) circle (0.5ex);
}}
}

\newcommand{\justanedge}{\text{%
\tikz[baseline]{\fill (0,0) circle (0.15ex);
\draw[-{Stealth[scale=0.7]}] (0,0) -- (0.8ex,0.8ex);
\draw (0.7ex,0.7ex) -- (1ex,1ex);
\fill (1ex,1ex) circle (0.15ex);
}}
}

\newcommand{\edgewithloops}{\text{%
\tikz[baseline]{\draw (-0.4ex,-0.2ex) circle (0.5ex);
\fill (-0.2ex,0.2ex) circle (0.15ex);
\draw[-{Stealth[scale=0.7]}] (-0.2ex,0.2ex) -- (0.78ex,0.62ex);
\draw (0.5ex,0.5ex) -- (1.2ex,0.8ex);
\fill (1.2ex,0.8ex) circle (0.15ex);
\draw (1.4ex,1.2ex) circle (0.5ex);
}}
}

\newcommand{\terminalgraph}{\text{%
\tikz[baseline]{\draw (0.5ex,0.5ex) circle (0.5ex);
\fill (0.5ex,0ex) circle (0.15ex);
}}
}

\newcommand{\fulltwo}{\text{%
\tikz[baseline]{\draw (-0.4ex,-0.2ex) circle (0.5ex);
\fill (-0.2ex,0.2ex) circle (0.15ex);
\draw[-{Stealth[scale=0.7]}] (-0.2ex,0.2ex) arc (180:100:0.8ex);
\draw (-0.2ex,0.2ex) arc (180:45:0.8ex);
\fill (1.2ex,0.8ex) circle (0.15ex);
\draw[-{Stealth[scale=0.7]}] (1.2ex,0.8ex) arc (0:-85:0.8ex);
\draw (1.2ex,0.8ex) arc (0:-135:0.8ex);
\draw (1.4ex,1.2ex) circle (0.5ex);
}}
}

\newcommand{\reflexiveparallel}{\text{%
\tikz[baseline]{\draw (-0.4ex,-0.2ex) circle (0.5ex);
\fill (-0.2ex,0.2ex) circle (0.15ex);
\draw[-{Stealth[scale=0.7]}] (-0.2ex,0.2ex) arc (180:100:0.8ex);
\draw (-0.2ex,0.2ex) arc (180:45:0.8ex);
\fill (1.2ex,0.8ex) circle (0.15ex);
\draw[-{Stealth[scale=0.7]}] (-0.2ex,0.2ex) arc (-135:-30:0.8ex);
\draw (-0.2ex,0.2ex) arc (-135:0:0.8ex);
\draw (1.4ex,1.2ex) circle (0.5ex);
}}
}

\newcommand{\suprisereverseedge}{\text{%
\tikz[baseline]{\draw (-0.4ex,-0.2ex) circle (0.5ex);
\fill (-0.2ex,0.2ex) circle (0.15ex);
\draw[-{Stealth[scale=0.7]}] (-0.2ex,0.2ex) arc (180:100:0.8ex);
\draw (-0.2ex,0.2ex) arc (180:45:0.8ex);
\draw (-0.2ex,0.2ex) -- (1.2ex,0.8ex);
\draw[-{Stealth[scale=0.7]}] (1.2ex,0.8ex) -- (-0.2ex,0.2ex);
\fill (1.2ex,0.8ex) circle (0.15ex);
\draw[-{Stealth[scale=0.7]}] (-0.2ex,0.2ex) arc (-135:-30:0.8ex);
\draw (-0.2ex,0.2ex) arc (-135:0:0.8ex);
\draw (1.4ex,1.2ex) circle (0.5ex);
}}
}

\title{Lassos: Pushing Tree Decompositions Forward Along Homomorphisms}
\author{Benjamin Merlin Bumpus\thanks{Instituto de Matemática e Estatística, Universidade de São Paulo. Rua do Matão, 1010 — 05508–090, São
Paulo, SP, Brasil.} \and James Fairbanks\thanks{University of Florida, Computer \& Information Science \& Engineering, Florida, USA.} \and Will J.\ Turner\thanks{TU Freiberg, Mathematics and Computer Science, Freistaat Sachsen, Deutschland.}}

\date{\today}

\begin{document}

\maketitle

\begin{abstract}
It is folklore that tree-width is monotone under taking subgraphs (i.e. \textit{injective} graph homomorphisms) and contractions (certain kinds of \textit{surjective} graph homomorphisms). However, although tree-width is obviously not monotone under \textit{any} surjective graph homomorphism, it is not clear whether contractions are canonically the only class of surjections with respect to which it is monotone. Under the requirement that the decomposition shape must be preserved, we prove that this is indeed the case. 

Our results provide a framework for answering questions of this sort for many other kinds of combinatorial data structures (such as directed multigraphs, hypergraphs, Petri nets, circular port graphs, half-edge graphs, databases, simplicial sets etc.) for which natural analogues of tree decompositions can be defined. Furthermore and of independent interest, we prove these results by introducing the notion of a \textit{lasso}, a generalization of contractions of graphs to arbitrary categories with pushouts of monomorphisms.

\let\thefootnote\relax\footnote{-- Authors Bumpus and Fairbanks acknowledge funding from the DARPA ASKEM and Young Faculty Award programs through grants HR00112220038 and W911NF2110323.}\addtocounter{footnote}{-1}\let\thefootnote\svthefootnote
\end{abstract}

\section{Introduction}

\body{Given any graph invariant $\mu$, it is always fruitful to enquire how $\mu$ behaves with respect to graph containment relations. This is often formalized as the question of determining whether the invariant $\mu$ gives rise to an order-preserving (or order-reversing) function $\mu \colon (\mathbb{G}, \triangleleft) \to (\mathbb{N}, \leq)$ from the poset of all graphs $(\mathbb{G}, \triangleleft)$ (viewed under some given order $\triangleleft$) to the natural numbers.}

\body{Graph invariants are often defined in terms of certain families of certifying objects which attest to the fact that a given graph attains some desired property. For instance, for any graph $G$, one has that: (1) $G$ has chromatic number at most $k$ if and only if it admits a proper coloring with at most $k$ colors; or (2) $G$ has tree-width at most $k$ if and only if it admits a tree decomposition of width at most $k$. }

\body{Often the relationship between graph invariants and containment relations can be furthermore carried over to a notion of monotonicity for the certificates for the invariant whereby one has a \textit{function} mapping the certificates of one graph to the certificates on another. For instance, if $H$ is a subgraph of a graph $G$, then: (1) every proper coloring of $G$ induces a proper coloring of $H$ and (2) every tree decomposition of width at most $k$ of $G$ induces on $H$ a decomposition of the same kind. }

\body{Thus the relationship between the sets of certificates of an invariant $\mu$ and some given containment relation $\triangleleft$ is best studied not through morphisms of posets of the form $\mu \colon (\mathbb{G}, \triangleleft) \to (\mathbb{N}, \leq)$, but instead through the notion of a \textit{functor} $M \colon \cat{Grph} \to \cat{Set}$ from the category\footnote{We assume the reader is familiar with the notion of categories, functors and natural transformations. Riehl's textbook~\cite{riehl2017category} is an excellent reference; Bumpus' thesis~\cite[Section 3.2]{bumpus2021generalizing} provides a very basic introduction which only requires minimal background in graph theory.} of graphs and graph homomorphisms $\cat{Grph}$ to the category of sets.}

\body{
In the case of tree decompositions, it has been recently shown~\cite{structured-decompositions} that there is a functor $\mathsf{MDgm} \colon \cat{Grph}^{\op} \to \cat{Set}$ taking each graph to the set of all of its decompositions (we recall this in Lemma~\ref{lemma: monic diagram presheaf}). This functor is contravariant, meaning that it reverses the direction of arrows: given any graph homomorphism \[ G \xrightarrow{f} H\] there is a function \[\mathsf{MDgm}(G) \xleftarrow{\mathsf{MDgm}(f)} \mathsf{MDgm}(H) \] sending each decomposition of $H$ (i.e. an element of $\mathsf{MDgm}(H)$) to a decomposition of $G$ (i.e. an element of $\mathsf{MDgm}(G)$). 
}

\body{
As is well-known in graph theory~\cite{Diestel2010GraphTheory}, it is easy to obtain such a mapping when the morphism $f$ is injective (i.e. $G$ is a subgraph of $H$). To see this, for any decomposition $d$ of $H$, observe that one obtains a decomposition of $G$ by intersecting each bag of $d$ with the image of $G$ under $f$.}

\body{
In contrast, to the best of our knowledge, it was not formalized until recently that one can obtain a decomposition of $G$ from a decomposition of $H$ given \textit{any} morphism $f \colon G \to H$, regardless of whether $f$ is injective or not. This is done by passing from "intersections" to the more general, categorical notion of a \textit{pullback}: given any tree decomposition $d$ of $H$, one can pull $d$ back along \textit{any} morphism $f \colon G \to H$ to obtain a tree decomposition of $G$ of the \textit{same shape} (this result, stated not just for graphs, but in much greater generality, is due to Bumpus, Kocsis, Master and Minichiello~\cite{structured-decompositions} and is also recalled in the present paper as Lemma~\ref{lemma: monic diagram presheaf}) (see also Figure~\ref{fig:example-of-pulling-back-a-decomp}).}

\begin{figure}
    \centering
    \includegraphics[width=.3\textwidth]{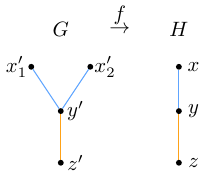}
    \caption{A trivial example of pulling a tree decomposition back along a graph homomorphism which is not injective. We are given two graphs $G$ and $H$, a graph homomorphism $f \colon G \to H$ and a tree decomposition of $H$ highlighted by the two colors of the edges of $H$ (the tree decomposition in question consists of two bags $\{x,y\}$ and $\{y,z\}$ with adhesion $\{y\}$). By pointwise pullback, one obtains a tree decomposition of $G$ consisting of the two bags $\{x'_1, x'_2, y'\}$ and $\{y', z'\}$.}
    \label{fig:example-of-pulling-back-a-decomp}
\end{figure}

\body{
Thus, since one can pull any decomposition of a graph $H$ back along \textit{any} morphism $f \colon G \to H$ to obtain a decomposition of $G$, it is natural to ask whether one can dually push decompositions \textit{forward}. In other words we ask: ``for which morphisms $f \colon G \to H$ can we produce a decomposition for $H$ starting with \textit{any} decomposition of $G$?''
}

\body{
Any graph theorist comfortable with tree decompositions would immediately guess that, if $f$ in the preceding paragraph is a \textit{contraction map}\footnote{Recall that a surjective graph homomorphism $f \colon G \twoheadrightarrow H$ is a \define{contraction map} if the preimage of any vertex in $H$ is a connected subgraph in $G$.} then we can always obtain a decomposition of $H$ from any decomposition $d$ of $G$ and moreover this decomposition will be of the same shape as $d$. This is well-known~\cite{Diestel2010GraphTheory} and the intuitive idea is that one simply identifies any two vertices in any bag of $d$ if those vertices are identified by $f$. 
}

\body{
It is thus natural to ask whether there are any other classes of surjective graph homomorphisms along which one can push a decomposition while leaving the shape of it unaltered. Here we answer this question negatively and show that contraction maps are the canonical class of maps with respect to which one can push decompositions \textit{forwards} while not altering the decomposition shape. Furthermore, we generalize this result far beyond graphs and show that it holds for many other kinds of combinatorial objects (specifically the objects of any topos, quasitopos or any sufficiently well-behaved locally cartesian-closed category). Some examples of these, among others, include: sets, symmetric graphs, directed graphs, directed multigraphs, hypergraphs, Petri nets, databases, simplicial sets, circular port graphs and half-edge graphs.}

\body{
To work at such a level of abstraction, we employ \textit{structured decompositions}~\cite{structured-decompositions}, recent category-theoretic generalization of graph decompositions which can be defined in any category. With lots of hand waving, structured decompositions are straightforward. One starts with a graph $J$ -- the \textit{shape} of the decomposition -- and a category $\cat{C}$ of things one might want to decompose. Then a structured decomposition is an assignment $d$ of objects of $\cat{C}$ to each vertex (a `\textit{bag}' of the decomposition) and edge (an `\textit{adhesion}' of the decomposition) of $J$ together with homomorphisms specifying how the adhesions nestle into their corresponding bags. Waving our hands a little slower, a structured decomposition $d$ is a special kind of diagram in $\cat{C}$ where one thinks of $d$ as \textit{decomposing} an object $c \in \cat{C}$ whenever $\colim d = c$. All of this will be treated formally in Definition~\ref{def:structured-decomposition}; for now however, a picture suffices (q.v. Figure~\ref{fig:str-decomp-example}). 
}

\begin{figure}[h]
    \centering
    \includegraphics[width=.9\textwidth]{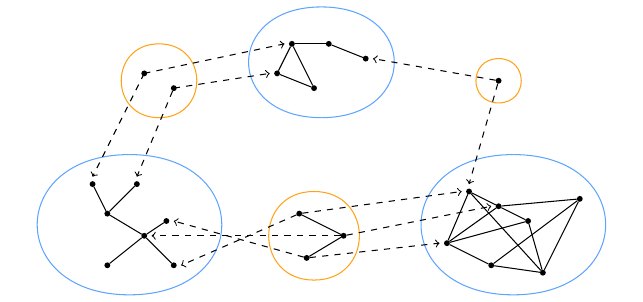}
        \caption{A structured decomposition of graphs whose shape is a triangle (i.e. a complete graph \(K_3\) on three vertices). The bags (circled in blue) and the adhesions (circled in gold) are related by spans of injective graph homomorphisms: these are drawn componentwise where each vertex of an adhesion is sent to a vertex in a relevant bag.}
    \label{fig:str-decomp-example}
\end{figure}

\body{
To obtain our results, we embrace a light category theoretic perspective. Peering through this lens, one finds that \textit{every} contraction map of graphs $f \colon G \twoheadrightarrow H$ arises as a pushout along a subobject of $G$ as in the following diagram. 
\begin{equation}\label{diagram:categorical-def-of-contractions-of-graphs}
\begin{tikzcd}
	K & G \\
	{\mathsf{cc}(K)} & H
	\arrow["g", hook, from=1-1, to=1-2]
	\arrow[two heads, from=1-1, to=2-1]
	\arrow["f", two heads, from=1-2, to=2-2]
	\arrow[from=2-1, to=2-2]
	\arrow["\lrcorner"{anchor=center, pos=0.125, rotate=180}, draw=none, from=2-2, to=1-1]
\end{tikzcd}
\end{equation}
This is to be read as follows: one first selects a subgraph $K$ of $G$, then one computes its connected components $\mathsf{cc}(K)$ (viewed as a reflexive\footnote{A graph is \define{reflexive} if every vertex has a loop-edge.} discrete graph) and then the pushout constructs $H$ by identifying each connected component of $K$ in $G$. Notice that, generalizing injective maps to \textit{monomorphisms} and surjective maps to \textit{epimorphisms}, the diagram above makes sense in \textit{any category} having enough pushouts. 
}

\body{
This category-theoretic take on the definition of a contraction map suggests a systematic way for \textit{defining} classes of epimorphisms\footnote{Notice that in the special case of sets (or graphs, for that matter) every surjective function can be written as a quotient over the relation given by an injective function. Translated to  ``abstract nonsense'', this is just saying that every epimorphism is \textit{regular} in $\cat{Set}$~\cite[Example 7.72]{adamek1990abstract}.}: one first defines a functor $\La \colon \cat{Grph} \to \cat{Grph}$ taking each graph $K$ to a graph $\La K$ which is to be thought of as playing the role of connected components in the previous paragraph and then one finds a natural transformation $\eta \colon \id_{\cat{Grph}} \to \La$ whose components are epimorphisms of the form $\eta_x \colon X \twoheadrightarrow \La X$. This defines a class of epimorphisms in $\cat{Grph}$ because, since pushouts preserve epimorphisms, one obtains an epimorphism $G \twoheadrightarrow G +_K \La K$ by pushout as we did in Diagram~\ref{diagram:categorical-def-of-contractions-of-graphs}. These observations lead us to the following definition in any category $\cat{C}$ which will be our main subject of study.
}

\begin{definition}\label{def:lasso}
    A \define{lasso} on $\cat{C}$ is a pair $(\La \colon \cat{C} \to \cat{C}, \eta \colon \id_{\cat{C}} \Rightarrow \La)$
    where: 
    \begin{enumerate}[label=\textbf{(L\arabic*)}]
        \item $\La$ is a functor that preserves pushouts of monomorphic spans\footnote{A \define{span} is a diagram of the form $A\leftarrow C\rightarrow B$.}; and \label{cond:lasso-1}
        \item $\eta$ is a natural transformation, all of whose components are epimorphisms. \label{cond:lasso-2}
    \end{enumerate}
\end{definition}

\body{One can verify that connected components are an example of a lasso on the category of graphs (see Section~\ref{sec:cat_of_lassos}). The following is a trivial example of a lasso on any category. }

\begin{example}
    Let $\cat{C}$ be a category. The \define{trivial lasso} on $\cat{C}$ is the lasso $(\id_{\cat{C}}, \id_{\id_{\cat{C}}})$ composed of the identity functor and the identity natural transformation.
\end{example}

\body{
As we have been foreshadowing, one can abstract the category theoretic definition of a contraction of graphs to obtain the notion of a contraction with respect to a choice of lasso.  
}

\begin{definition}\label{def:lasso-contraction}
    A category $\cat{C}$ \define{admits contractions with respect to a lasso} $(\La, \eta)$ if the following pushout exists for all monomorphisms $f \colon X \hookrightarrow Y$ in $\cat{C}$. 
\[\begin{tikzcd}
	X & Y \\
	{\La X } & {Y\contract{f}}
	\arrow["f", hook, from=1-1, to=1-2]
	\arrow["{\eta_X}"', two heads, from=1-1, to=2-1]
	\arrow[from=2-1, to=2-2]
	\arrow[two heads, from=1-2, to=2-2]
	\arrow["\lrcorner"{anchor=center, pos=0.125, rotate=180}, draw=none, from=2-2, to=1-1]
\end{tikzcd}\]
We call the pushout of $f$ and $\eta_X$ the \define{contraction of $Y$ along $f$} and denote it by $Y\contract{f}$. 
\end{definition}

\body{
Definition~\ref{def:lasso-contraction} illuminates why Condition~\ref{cond:lasso-2} is necessary in our definition of lassos (if not absolutely natural). However, we have not yet clarified the reason for the first condition -- namely the preservation of monic pushouts. Spelling this out, Condition~\ref{cond:lasso-1} prescribes that, if we are given any monic span 
\(\begin{tikzcd}
	A & C & B
	\arrow["a"{description}, hook', from=1-2, to=1-1]
	\arrow["b"{description}, hook, from=1-2, to=1-3]
\end{tikzcd}\)
with pushout $X$ in~$\cat{C}$, then the pushout of its image 
\(\begin{tikzcd}
	{\La A} & {\La C} & {\La B}
	\arrow["{\La a}"{description}, from=1-2, to=1-1]
	\arrow["{\La b}"{description}, from=1-2, to=1-3]
\end{tikzcd}\) under $\La$ is $\La X$. Since a structured decomposition is nothing other than a diagram that is drawn by piecing spans of monomorphisms together, it follows by induction that Condition~\ref{cond:lasso-1} implies the preservation of all finite tree-shaped structured decompositions. Now, for any object $C$ in $\cat{C}$, we can contract $C$ along a map $C\hookrightarrow C$ (itself) to obtain $\La C$ as a contraction; hence, without the preservation of monic pushouts, we have fallen at the first hurdle, let alone for more complicated contractions. 
}

\body{In Section~\ref{sec:monotone} (q.v. Paragraph~\ref{para:proof-of-thm:main}) we obtain the following result which shows that decompositions always interact nicely with lasso-contractions whenever one is working in what we call a \textit{mono-strong} category (c.f. Definition~\ref{def: mono-strong}). Examples of such categories include categories of: sets, symmetric graphs, directed graphs, directed multigraphs, hypergraphs, Petri nets, circular port graphs half-edge graphs, databases, simplicial sets etc..}

\begin{theorem}\label{thm:main}
Suppose $C$ is a mono-strong category admitting pullback-stable colimits of shape $\smallint T \to \cat{C}$ for all trees $T$. If $d \colon \smallint T \to \cat{C}$ be a \textit{tree-shaped} structured decomposition with colimit $Y$, then, for any lasso $(\La, \eta)$ on $\cat{C}$ and any monomorphism $f \colon X \hookrightarrow Y$, we have that:
\begin{enumerate}
    \item we can push the diagram $d$ forwards along the $(\La, \eta)$-contraction $Y \twoheadrightarrow Y\contract{f}$ to obtain a diagram $d\contract{f} \colon \smallint T \to \cat{C}$ of the same shape;
    \item the width\footnote{The \textit{width} of a tree decomposition of graphs is the maximum size of any of the pieces (often called \textit{bags} or \textit{parts}) of the decomposition. This notion admits an appropriate categorial generalization as shown in~\cite{structured-decompositions}.} of $d/f$ is at most that of $d$; and
    \item $\colim (d \contract{f}) = (\colim d)\contract{f} = Y\contract{f}$.
\end{enumerate}
\end{theorem}

\begin{remark}
    Theorem~\ref{thm:main} can be adapted to structured decompositions of \textit{arbitrary shape} by imposing additional requirements on the definition of a lasso: in addition to preserving monic pushouts, one furthermore demands the preservation of colimits (qv. the notion of \textit{strong lassos} given in Definition~\ref{def:strong-lasso} and Proposition~\ref{prop:local-contractions-give-correct-colimit}).  
\end{remark}

\body{
Our categorical treatment allows us to employ elementary arguments to conclude the following canonicity result. We will detail the arguments in Section~\ref{sec:cat_of_lassos} (q.v. Paragraph~\ref{para:proof-of-thm:canonicity} in particular). In turn, this implies canonicity of the notion of contraction maps (Corollary~\ref{corollary:contraction-canonical}).
}

\begin{theorem}\label{thm:canonicity}
The connected component lasso is the only nontrivial lasso on the category $\cat{Grph}$ of directed multigraphs.
\end{theorem}

\begin{corollary}\label{corollary:contraction-canonical}
In $\cat{Grph}$ the class of all epimorphisms with respect to which tree-width is monotone is precisely the class of contractions. 
\end{corollary}

\section{Preliminaries}
\body{
In this section we will recall some basic categorical notions which we will need in the rest of the article. These include  image factorizations in arbitrary categories, the epi- and mono-triangle lemmas and a brief review of the category of graphs viewed as a functor category. Finally we will recall some examples of categories that admit pullback-stable colimits and show how many familiar kinds of combinatorial data structures (e.g. categories of: sets, symmetric graphs, directed graphs, directed multigraphs, hypergraphs, Petri nets, circular port graphs half-edge graphs, databases, simplicial sets etc.) assemble into such categories.
}

\body{
As is common in category theory, we view graphs as functors \(G \colon \cat{SGr} \to \cat{Set}\) where \(\cat{SGr}\) is a category with only two objects called \(E\) and \(V\) and two non-identity morphisms \(s, t \colon E \to V\) (the `source' and `target' maps). The functor category \(\cat{Set}^{\cat{SGr}}\) -- which we denote as $\cat{Grph}$ for convenience -- has directed graphs (allowing loops and parallel edges) as objects and graph homomorphisms as arrows. To spell this out, notice that a functor \(G \colon \cat{SGr} \to \cat{Set}\) consists of: 
\begin{itemize}
    \item a set of vertices \(G(V)\) and a set of edges \(G(E)\)
    \item a source function $G(s) \colon G(E) \to G(V)$ and a target function $G(s) \colon G(E) \to G(V)$ which map each edge of $G$ to its source and target vertices. 
\end{itemize}
Being a functor category, the arrows in $\cat{Set}^{\cat{SGr}}$ are natural transformations. It is easy to verify that a natural transformation \(\eta \colon G \Rightarrow H\) defines a graph homomorphism from $G$ to $H$. 
}\label{para: def graph}

\body{
As pointed out by Spivak~\cite{SPIVAK201231}, a functor category $\cat{Set}^\cat{C}$ (where $\cat{C}$ is finite) can be seen as a category of combinatorial data: the category $\cat{C}$ is playing the role of a schema that defines the kind of data structure one might be interested in (just as the category $\cat{SGr}$ was used to define $\cat{Grph}$ in the previous paragraph). For example, the following is the schema category used to define a category of Petri nets. 
\[\begin{tikzcd}[row sep = small]
	& {\mathsf{Input}} \\
	{\mathsf{Token}} & {\mathsf{Species}} & {\mathsf{Transition}} \\
	& {\mathsf{Output}}
	\arrow[from=1-2, to=2-2]
	\arrow[from=1-2, to=2-3]
	\arrow[from=2-1, to=2-2]
	\arrow[from=3-2, to=2-2]
	\arrow[from=3-2, to=2-3]
\end{tikzcd}\]
Such functor categories (also known as `$\cat{C}$-sets'~\cite{SPIVAK201231, Patterson2022categoricaldata}) all share certain properties in common with the category of graphs and in fact they all form \textit{presheaf toposes}. Thus, rather than proving our results for each of these categories in turn, we will argue at the level of what we will call \textit{mono-strong} categories (c.f. Definition~\ref{def: mono-strong}). Roughly, one should think of such a category as one where pushouts and pullbacks behave nicely with each other. 
}\label{para: def petri}

\begin{example}\label{example:pushout}
Recall that a \textit{pushout} in a category $\cat{C}$ is a colimit of a diagram of shape $\bullet \leftarrow \bullet \rightarrow \bullet$. Roughly this is to be thought of as the operation of `gluing' two objects (the feet of the span) together along a third mediating object (the apex of the span). In the case of graphs, the pushout of a span $G_1 \xleftarrow{f_1} G_{12} \xrightarrow{f_2} G_2$ is the cospan $G_1 \rightarrow G_1 +_{G_{12}} G_{2} \leftarrow G_2$ where the graph $G_1 +_{G_{12}} G_{2}$ is defined as the quotient of the disjoint union $G_1 + G_2$ of $G_1$ and $G_2$ under the equivalence relation which identifies two vertices (resp. edges) $x$ and $y$  of $G_1 + G_2$ whenever there is a vertex (resp. edge) $w$ of $G_{12}$ such that $f_1(w) = x$ and $f_2(w) = y$. 
\end{example}

\body{
Pushouts are special cases of more general constructions called \textit{colimits}. The interaction of pushouts (or colimits more generally) with pullbacks is not always easy to determine. However, in many familiar cases such as toposes (e.g. categories of (co)presheaves c.f Paragraphs~\ref{para: def graph} and~\ref{para: def petri}), quasitoposes and even more generally any locally cartesian-closed category, one has a strong property of pullback-stable colimits given below.
}

\begin{definition}\label{def:pullback-stable}
    Let $\mathcal{J}$ be a class of diagrams in a category $\cat{C}$ with pullbacks and colimits of all diagrams in $\mathcal{J}$. One says that a $\cat{C}$ has pullback-stable colimits of type $\mathcal{J}$ if for any diagram $d$ in the class and any morphism $f \colon x \to \colim(d)$, the diagram $d'$ obtained by taking pointwise pullbacks of the colimit cocone of $d$ and $f$ satisfies $\colim(d') \cong x$.
\end{definition}

\body{
Structured decompositions, the category theoretic generalization of graph decompositions which we will use throughout this article, are most well-behaved in categories that have enough pullback-stable colimits.  
}

\begin{definition}[\cite{structured-decompositions}]\label{def:structured-decomposition}
For any graph $J$ one can define a \textit{category} $\int J$ as follows:
\begin{enumerate}
    \item each vertex $x$ of $J$ gives rise to an object of $\int J$;
    \item each edge $e$ of $J$ gives rise to an object of $\int J$;
    \item each edge $e = (x,y)$ in $J$ gives rise to precisely two arrows in $\int J$: one from $e$ to $x$ and one from $e$ to $y$ (arrows of this kind are the only non-identity arrows in $\int J$).
\end{enumerate} Fixing a base category $\cat{C}$, a $\cat{C}$-valued \define{structured decomposition} of shape $J$ is a diagram of the form \(d \colon \int J \to \cat{C}\) which assigns to each edge $e = (x,y)$ in $J$ a span $d(x) \hookleftarrow d(e) \hookrightarrow d(y)$ of monomorphisms (cf. Figure~\ref{fig:str-decomp-example} where $J$ is taken to be the graph $K_3$). 
$\cat{C}$-valued structured decompositions (of all shapes) assemble into a \define{category $\mathfrak{D}(\cat{C})$} where a morphism from a structured decomposition $d_1 \colon \int J_1 \to \cat{C}$ to a structured decomposition $d_2 \colon \int J_2 \to \cat{C}$ is a pair $(f,\gamma)$ of a functor $f$ and an natural transformation $\gamma\colon d_1\Rightarrow d_2f$ as in the following triangle.
\[\begin{tikzcd}
	{\int J_1} && {\int J_2} \\
	&& {} \\
	& {\cat{C}}
	\arrow["f", from=1-1, to=1-3]
	\arrow[""{name=0, anchor=center, inner sep=0}, "{d_1}"', from=1-1, to=3-2]
	\arrow["{d_2}", from=1-3, to=3-2]
	\arrow["\gamma", shorten <=13pt, shorten >=13pt, Rightarrow, from=0, to=1-3]
\end{tikzcd}\]
\end{definition}

\body{To aid intuition, we make an analogy to graph decompositions. Fix any decomposition $d \colon \int J \to \cat{C}$. For any vertex $x$ in $J$, think of the object $d(x)$ in $\cat{C}$ as a \textit{bag} of the decomposition. Given any edge $e = (x,y)$ of $G$, think of \(d(e)\) as an \textit{adhesion}  of the decomposition. The span $d(x) \hookleftarrow {d (e)} \hookrightarrow d(y)$ specifies how the adhesion fits into its corresponding bags. One says that $d$ \define{decomposes} an object $x$ if $\colim d = x$.
}

\body{
Other than categories with pullback-stable colimits and structured decompositions, we will also frequently use image factorizations. These are the categorical generalization to arbitrary categories of the notion of factorizing a function of sets (or a group homomorphism, etc.) through its image. 
}

\begin{definition}\label{def:image-factorization}
By a \define{factorization} of a morphism $X \xrightarrow{f} Y$, we mean any pair of composable morphisms $X \xrightarrow{g} Z \xrightarrow{h} Y$ such that $hg = f$. We will say that the factorization is \define{mono} if $h$ is a monomorphism. The \define{image factorization} of a morphism $f \colon X \to Y$, if it exists, is the mono-factorization denoted $X \xrightarrow{\mi f} \imageObj f \xrightarrow{\image f} Y$ which is initial in the category of mono-factorizations. 
\end{definition}

\body{
Definition~\ref{def:image-factorization} can be unpacked as follows: a mono-factorization $X \xrightarrow{\mi f} \imageObj f \xrightarrow{\image f} Y$ of a morphism $f$ is its \textit{image factorization} if, for any other mono-factorization $X \to Z \hookrightarrow Y$ of $f$ there is a \textit{unique} morphism $\imageObj f \to Z$ making the following diagram commute.
\[\begin{tikzcd}
	X && Y \\
	& {\imageObj f} \\
	& {\forall Z}
	\arrow["f", from=1-1, to=1-3]
	\arrow["{\image f}", hook, from=2-2, to=1-3]
	\arrow[hook, from=3-2, to=1-3]
	\arrow["{\mi f}", from=1-1, to=2-2]
	\arrow[from=1-1, to=3-2]
	\arrow["{\exists!}", dashed, from=2-2, to=3-2]
\end{tikzcd}\]
}

\body{
It is a folkloric fact (q.v. Riehl~\cite{riehl2017category}) that, if a category has all equalizers, then image factorizations consist of an epimorphism and a monomorphism respectively (i.e. they are epi-mono-factorizations). This is the case in the category of sets, the category of graphs, any functor category $\cat{Set}^\cat{J}$ and indeed any topos. Another feature common to many of these examples is that of being \textit{balanced}.
}

\begin{definition}\label{def:balanced}
    A category is \define{balanced} if, for any morphism, being both monomorphic and epimorphic implies being an isomorphim.
\end{definition}

\body{
It will often be very useful in the proofs of the next section to be able to deduce, given a commutative triangle, if one of the three arrows is an epimorphism (or a monomorphism) when it is known that the other two also are. Thus we recall the following folkloric lemma.  
}

\begin{lemma}[epi-triangle lemma]\label{lemma:epi-triangle}
In any category, suppose that $yx$ is a factorisation of some epimorphism $f$. Then $y$ is an epic.
\[\begin{tikzcd}
	& \bullet \\
	\bullet && \bullet
	\arrow["y", from=1-2, to=2-3]
	\arrow["x", from=2-1, to=1-2]
	\arrow["f", two heads, from=2-1, to=2-3]
\end{tikzcd}\]
\end{lemma}
\begin{proof}
For all $p,q \colon Y \to A$, if $py=qy$ then $pf = pyx = qyx = qf$. Since $y$ is epic, $p = q$.
\end{proof}

\body{\label{para:mono-triangle}
For ease of reference we note that by categorical duality Lemma~\ref{lemma:epi-triangle} also implies the following: if a monomorphism $f$ factors as $xy$, then $y$ is monic.
\[\begin{tikzcd}
	& \bullet \\
	\bullet && \bullet
	\arrow["x"', from=1-2, to=2-1]
	\arrow["y"', from=2-3, to=1-2]
	\arrow["f"', hook, from=2-3, to=2-1]
\end{tikzcd}\]
}

\begin{lemma}\label{image_epi_precomp}
    Let $A\xrightarrow{g}B\xrightarrow{f}C$ be two arrows in a balanced category  $\cat{C}$ admitting image factorisations. If $g$ is epic, then $\imageObj f\circ g$ and $\imageObj f$ are isomorphic.
\end{lemma}
\begin{proof}
    By the universal property of images, we get a unique map $q\colon\imageObj f\rightarrow\imageObj f\circ g$ making the following commute.
\[\begin{tikzcd}
	A & B && C \\
	&& {\imageObj f} \\
	&& {\imageObj{f\circ g}}
	\arrow["g", two heads, from=1-1, to=1-2]
	\arrow["{\mi f\circ g}"', two heads, from=1-1, to=3-3]
	\arrow["f", from=1-2, to=1-4]
	\arrow["{\mi f}"', two heads, from=1-2, to=2-3]
	\arrow["{\im f}", hook, from=2-3, to=1-4]
	\arrow["{\exists ! q}"{description}, dashed, from=2-3, to=3-3]
	\arrow["{\im f\circ g}"', hook, from=3-3, to=1-4]
\end{tikzcd}\]
    Applying Results~\ref{lemma:epi-triangle}, \ref{para:mono-triangle}, yields that $q$ is an isomorphism.
\end{proof}

\body{
Throughout the rest of the article we will be always working with categories which we call $\mathcal{J}$-strong. These include all toposes and hence all of the combinatorial examples we mentioned earlier. 
}

\begin{definition}\label{def: mono-strong}
    Let $\mathcal{J}$ be a class of diagrams in a category $\cat{C}$, we say that $\cat{C}$ is \define{$\mathcal{J}$-strong} if the following conditions are satisfied: 
    \begin{enumerate}
        \item $\cat{C}$ is balanced,
        \item $\cat{C}$ has pullbacks and equalizers and
        \item $\cat{C}$ has pullback-stable colimits of type $\mathcal{J}$.
    \end{enumerate}
    We say that $\cat{C}$ is \define{mono-strong} if $\mathcal{J}$ is the class of all monic diagrams (i.e. diagrams whose image consists only of monomorphisms).
\end{definition}

\body{
Notice that, by having all equalizers, \define{$\mathcal{J}$-strong} categories always admit epi-mono-factorizations. 
}

\section{The category of lassos}
\label{sec:cat_of_lassos}
\body{
To study the space of possible lassos on a given category, we introduce the following notion.
}
\begin{definition}[The category of lassos]
    Let $\cat{C}$ be a category. The \define{category of lassos on $\cat{C}$} -- denoted $\cat{Lasso}(\cat{C})$ -- has as objects the lassos on $\cat{C}$ and its morphisms $f:(\La,\eta)\rightarrow(\La',\eta')$ are given by collections of epimorphisms\footnote{Here we specify that the maps $f_A$ have to be epimorphisms but observe that, by Lemma~\ref{lemma:epi-triangle}, this is already guaranteed.} $(f_A \colon \La A \twoheadrightarrow \La' A)_{A\in\cat{C}}$ indexed by the objects of $\cat{C}$ making the following triangle commute.
\[\begin{tikzcd}
	& A \\
	{\La A} && {\La'A}
	\arrow["{\eta_A}", two heads, from=1-2, to=2-1]
	\arrow["{\eta'_A}"', two heads, from=1-2, to=2-3]
	\arrow["{f_A}"', two heads, from=2-1, to=2-3]
\end{tikzcd}\]
    The composition of morphisms is given by the composition of the underlying maps.
\end{definition}

\body{
    Let $(\La,\eta)$ and $(\La',\eta')$ be lassos on some category $\cat{C}$. Let $\eta'\circ \eta\colon\id_\cat{C}\Rightarrow\La'\circ\La$ be the vertical composition of $\eta'$ after $\eta$. That is, elementwise we have $(\eta'\circ\eta)_X=(\eta')_{\La X}\circ(\eta)_X$. It is easy to check that the pair $(\La'\circ\La,\eta'\circ\eta)$ is a lasso on $\cat{C}$. In this context, we write $(\La',\eta')\circ(\La,\eta):=(\La'\circ\La,\eta'\circ\eta)$ and call it the \define{composition} of $(\La,\eta)$ and $(\La',\eta')$. 
}

\body{
    Observe that the collection $\eta_{\La X}$ is a morphism in the category of lassos from $(\La,\eta)$ to $(\La',\eta')\circ(\La,\eta)$. Further observe that the trivial lasso is the identity with respect to the composition of lassos. In Appendix~\ref{appendix:double}, we prove that the category of lassos can be viewed as a single object double category.
}

\body{
    It is easy to verify that, for a given lasso $(\La,\eta)$, the collection $\eta_X$ is the unique morphism from the trivial lasso to $(\La,\eta)$ in the category of lassos. Hence the trivial lasso is the initial object in the category of lassos.
}

\subsection{The special case of Graphs: canonicity of contractions.}

\body{
In this section, we will see that there is exactly one nontrivial lasso on $\cat{Grph}$. Firstly, we prove a few results that let us map out the category of lassos.
}

\begin{proposition}\label{prop:cc_coco}
There is a cocontinuous functor $\cat{cc} \colon \cat{Grph} \to \cat{Grph}$ which maps any graph $G$ to the graph $\cat{cc}(G)$ obtained by quotienting the vertex set of $G$ by the connectivity relation.
\end{proposition}
\body{\label{para:coequaliser_dfn}
    In the proof of the above proposition, we require the coequaliser of graphs. A \define{coequaliser} is the colimit of a diagram of the form
    \[\begin{tikzcd}
	H & G
	\arrow["a", curve={height=-6pt}, from=1-1, to=1-2]
	\arrow["b"', curve={height=6pt}, from=1-1, to=1-2]
    \end{tikzcd}\]
    For graphs $H$ and $G$, along with a pair of graph homomorphisms $a,b\colon H\to G$, the coequaliser is a graph $C$ obtained from $G$ by the following identifications: vertices~$u$ and~$v$ are identified if they are the images $a(w)$ and $b(w)$ of a single vertex $w$ in $H$. Edges $e$ and $f$ are identified if they are the images $a(g)$ and $b(g)$ of a single edge $g$ in $H$.
}
\begin{proof}[Proof of Proposition~\ref{prop:cc_coco}]
Let $\cat{cc}$ denote the map that sends each graph to the graph obtained by quotienting its vertex set by the connectivity relation. Since graph homomorphisms preserve the connectivity relation, we can extend $\cat{cc}$ to a functor by sending each graph homomorphism $f\colon G\to H$ to the map from $\cat{cc}(G)$ to $\cat{cc}(H)$ that sends each vertex to the the representative of its equivalence class in the image of $f$. This map $\cat{cc}(f)$ acts on the edge set of $\cat{cc}(G)$, which is the same as the edge set of $G$, identically to $f$. Given a graph $G$, let $\pi_G:G\rightarrow \cat{cc}(G)$ be the map corresponding to the `action of $\cat{cc}$'. That is, the map sending each vertex to its representative in the quotient and acting as the identity on edges. Observe that the maps $\pi_G$ assemble into a natural transformation $\pi\colon\id_\cat{Grph}\Rightarrow\cat{cc}$.

We are required to show that $\cat{cc}$ preserves all coproducts and coequalisers. For coproducts, this follows immediately because they cannot change the connectivity relation. For coequalisers, let $a,b:H\rightarrow G$ be an arbitrary pair of graph homomorphisms. Consider the diagram $\cat{cc}(a),\cat{cc}(b)\colon\cat{cc}(H)\to\cat{cc}(G)$. Observe that if a pair of vertices $u$ and $v$ in $G$ are the images $a(w)$ and $b(w)$ of a single vertex in $H$, then $\pi_G(u)$ and $\pi_G(v)$ in $\cat{cc}(G)$ are the images of $\pi_G(a(w))$ and $\pi_G(b(w))$, which, by the naturality of $\pi$, are in turn equal to $\cat{cc}(a)(\pi_H(w))$ and $\cat{cc}(b)(\pi_H(w))$, respectively. Hence the identification of $u$ and $v$ in the coequaliser of $a,b:H\rightarrow G$ implies the identification of $\pi_G(u)$ and $\pi_G(v)$ in the image of the coequaliser under $\cat{cc}$. For the converse, suppose that $\pi_G(x)$ and $\pi_G(y)$ are the images of a single vertex $\pi_G(z)$ (for some choice of representative $z$ from $H$). Then there are vertices $x',y'$ and $z'$ in the connected components containing $x,y$ and $z$, respectively, so that $z'=a(x')=b(y')$. Since graph homomorphisms preserve the connectivity relation, we get that there is one component of the coequaliser of $a,b:H\rightarrow G$ containing the vertices $a(x)$ and $b(y)$. Hence these have the same image in the image of the coequaliser under $\cat{cc}$. This completes our claim that $\cat{cc}$ preserves coequalisers.
\end{proof}

\body{
We call the lasso defined by the above the \define{connected components lasso} on $\cat{Grph}$ and we denote its associated natural transformation by $\pi\colon\id_{\cat{Grph}}\Rightarrow\cat{cc}$.
}

\begin{figure}
    \centering
    \includegraphics[width=.5\textwidth]{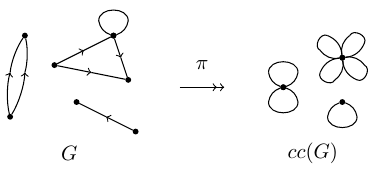}
    \caption{An illustration of the action of the connected components lasso on a particular graph $G$}
    \label{fig:illustration_of_lasso}
\end{figure}

\begin{lemma}\label{lemma:condition_for_double_loop_particle}
    Let $\Fa$ be a functor from $\cat{Grph}$ to $\cat{Grph}$ along with a natural transformation $\alpha:\cat{id}_{\cat{Grph}}\Rightarrow \Fa$ whose components are epic. Let $\twoloop$ be the graph with one vertex and two loops. If for some graph $G$, we have that $\alpha_G$ identifies a pair of edges in $G$, then $\alpha_{\twoloop}$ maps $\twoloop$ to the terminal graph\footnote{The terminal graph is the terminal object in $\cat{Grph}$, consisting of one vertex with a single loop.}.
\end{lemma}
\begin{proof}
Suppose there is a graph $G$ containing a pair of edges $e$ and $e'$ such that $e$ and $e'$ are identified by $\alpha_G\colon G\to\Fa(G)$. Then let $f \colon G \to \twoloop$ denote a homomorphism which identifies all vertices in $G$ and maps $e$ and $e'$ to different loops in $\twoloop$. By the functoriality of $\Fa$ and the fact that the components of $\alpha$ are epic, we have that $\Fa(\twoloop)$ is the terminal graph. This situation is depicted in Figure~\ref{fig:loop_particle}.
\begin{figure}[h]
    \centering
    \includegraphics[width=.5\textwidth]{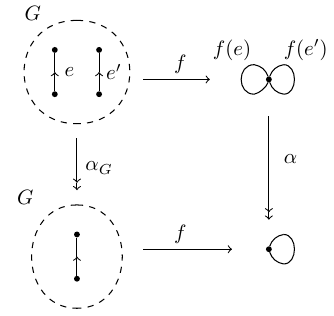}
    \caption{A depiction of the proof of Lemma~\ref{lemma:condition_for_double_loop_particle}}
    \label{fig:loop_particle}
\end{figure}
\end{proof}
\begin{lemma}\label{lemma:oh_no}
    Let $\Fa$ be a functor from $\cat{Grph}$ to $\cat{Grph}$ and suppose there is a natural transformation $\alpha:\cat{id}_{\cat{Grph}}\Rightarrow \Fa$ whose components are epic. Let $\twoloop$ be the graph with one vertex and two loops. If $F$ maps $\twoloop$ to the terminal graph, then $F$ does not preserve monic pushouts.
\end{lemma}
\begin{proof}
    Let $d$ be the diagram consisting of the monic span of two single vertices with loops over a vertex with no loop, as in Figure~\ref{fig:not_right_colim_1}. Observe that $\twoloop$ is the colimit of $d$. That is, a monic pushout. Further observe that, since the components of $\alpha$ are epic, the diagram $d$ is unchanged under $F$. Hence the colimit of $F(d)$ is $\twoloop$, which is not $F(\twoloop)$.
    \begin{figure}[h]
    \centering
    \includegraphics[width=.7\textwidth]{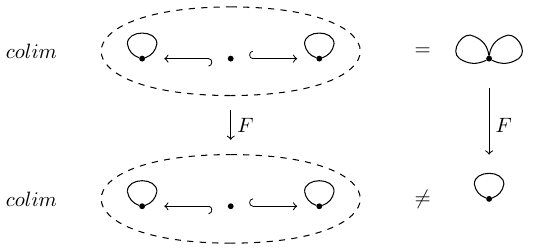}
    \caption{A depiction of the counterexample from Lemma~\ref{lemma:oh_no}}
    \label{fig:not_right_colim_1}
\end{figure}
\end{proof}
\body{
We say that a graph homomorphism is \define{vertex-trivial} if its vertex map is a bijection. We say that a lasso is \define{vertex-trivial} if all of the components of its natural transformation are vertex-trivial. We define \define{edge-triviality} analogously.
}

\begin{lemma}\label{lemma:cant_identify_edges}
    All lassos on $\cat{Grph}$ are edge-trivial.
\end{lemma}
\begin{proof}
Let $(\La,\eta)$ be a lasso on $\cat{Grph}$ and suppose for a contradiction that for some graph $G$, the map $\eta_G$ identifies a pair of edges in $G$. By combining Lemma~\ref{lemma:condition_for_double_loop_particle} and Lemma~\ref{lemma:oh_no}, we get that $\La$ does not preserve monic pushouts, a contradiction.
\end{proof}

\begin{corollary}\label{cor:must_identify_vertices}
    The trivial lasso is the unique vertex-trivial lasso on $\cat{Grph}$. \qed
\end{corollary}

\begin{lemma}\label{lemma:contraction_terminal}
    The connected components lasso $(\cat{cc},\pi)$ is the terminal object in the category of lassos on $\cat{Grph}$.
\end{lemma}
\begin{proof}
    Let $(\La,\eta)$ be a lasso on $\cat{Grph}$. We are required to show that there is a unique map from $(\La,\eta)$ to $(\cat{cc},\pi)$ in the category of lassos on $\cat{Grph}$.
    Suppose first that for every graph $G$ and every pair of vertices $u$ and $v$ in $G$ that lie in distinct connected components, we have that $\eta_G(u)$ and $\eta_G(v)$ are distinct. Then for each vertex $w$ in $\La G$, there is a unique vertex $f_G(w)$ in $\cat{cc}\;G$ so that for all $u\in\eta_G^{-1}(w)$ we have $\pi_G(u)=f_G(w)$. The function $f_G$ extends uniquely to a homomorphism~$\La G\twoheadrightarrow\cat{cc}\; G$ by sending every edge $xy$ to the single loop at the vertex $f_G(x)=f_G(y)$.

    Now suppose for a contradiction that there is some graph $G$ and a pair of vertices $x$ and $y$ in different connected components of $G$ so that $\eta_G$ identifies $x$ and $y$. Let $\twodiscrete$ be the graph with two vertices and a loop at each vertex. Let $h\colon G\rightarrow \twodiscrete$ be a homomorphism sending each of $x$ and $y$ to separate vertices. By the naturality of $\eta$ and the fact that~$\eta_{\twodiscrete}$ is epic, we have that $\La \twodiscrete$ must be a graph consisting of a single vertex. However, $\twodiscrete$ is the coproduct of two copies of a single vertex with a loop. Since coproducts are monic pushouts over an empty graph, we get a contradiction to the fact that $\La$ preserves monic pushouts.
\end{proof}

\begin{lemma}\label{lemma:basically_cc}
    Suppose that $(\La,\eta)$ is a lasso on $\cat{Grph}$ which is not vertex-trivial. Let $G$ be an arbitrary graph. Then $\La G$ has one vertex for every connected component of $G$ and $\eta_G\colon G\to\La G$ maps each vertex to the representative of its component. Furthermore, $\La G$ has no edges between distinct vertices.
\end{lemma}
\begin{proof}
    Let $G$ be a graph so that $\eta_G$ identifies some pair $x$ and $y$ of its vertices.
    Let~$\fulltwo$ be the graph consisting of two vertices, each having a loop, along with an edge in each direction between the two. Then there is a morphism $f\colon G\to \fulltwo$ sending each of $x$ and $y$ to a distinct vertex of $\fulltwo$. By naturality of $\eta$, we conclude that $\La \fulltwo$ is a graph consisting of a single vertex (but its number of loops is yet undetermined).

    Let $\justanedge$ be the graph consisting of a single edge. Observe that $\fulltwo$ can be constructed as the colimit of the following structured decomposition.
    \begin{center}
        \includegraphics[width=.7\textwidth]{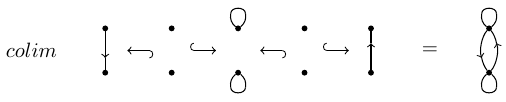}
    \end{center}
    Since $\La$ preserves monic pushouts, and hence preserves structured decompositions, the image of the above diagram under $\Fa$ should have $\La \fulltwo$ as a colimit. By Lemma~\ref{lemma:contraction_terminal}, the functor $\La$ cannot map the graphs having two disconnected vertices to a graph with only one vertex. Hence, in order to produce the right colimit,~$\La \justanedge$ must be the terminal graph.
    
    Now let $G$ be an arbitrary graph and observe that for every edge $e$ of $G$ there is a morphism $f\colon \justanedge\to G$ which maps $\justanedge$ onto $e$. By naturality, we must have that the endvertices of $e$ are identified by $\eta_G$. Since $e$ was chosen arbitrarily, all pairs of endvertices of edges are identified. By Lemma~\ref{lemma:contraction_terminal}, no vertices between connected components are identified by $\eta_G$ and we conclude that $\eta_G$ maps each vertex to a representative of its connected component.
    Finally, we note that there are no edges between vertices of $\La G$ since $\eta_G$ is epic and this would imply edges between connected components of $G$.
\end{proof}

\body{Now we prove Theorem~\ref{thm:canonicity}, which states that the connected components lasso is the only nontrivial lasso on the category of graphs.}

\body{
\begin{proof}[Proof of Theorem~\ref{thm:canonicity}]
    Let $(\La,\eta)$ be a nontrivial lasso on $\cat{Grph}$. By Lemma~\ref{lemma:cant_identify_edges}, the lasso~$(\La,\eta)$ is edge-trivial but not vertex-trivial. By Lemma~\ref{lemma:basically_cc}, we conclude that~$(\La,\eta)$ is the connected components lasso.
\end{proof}}\label{para:proof-of-thm:canonicity}

\body{
By the above analysis, the category of lassos on $\cat{Grph}$ is simply the following category of two objects.
\[\begin{tikzcd}
	{(\id,\id)} & {(\cat{cc},\pi)}
	\arrow[from=1-1, to=1-2]
\end{tikzcd}\]

\begin{remark}
    The category $\cat{Grph}$ is the category of directed multigraphs. If we restrict our attention to the category of undirected multigraphs, it is easy to see that all of the above results specialize and hence the category of lassos is essentially identical to the above.
\end{remark}

\subsection{Examples of lassos beyond $\cat{Grph}$}

\subsubsection{Reflexive graphs}

\body{
There are two equivalent ways to define the category of reflexive graphs. Firstly, one may take the objects to be the all graphs, but allow the morphisms to map edges to vertices. Secondly, we can modify the schema that defines $\cat{Grph}$ so that every vertex comes with a distinguished loop. We will use this second formulation. The category
\[\begin{tikzcd}
	E \\
	V
	\arrow["t", curve={height=-12pt}, from=1-1, to=2-1]
	\arrow["s"', curve={height=12pt}, from=1-1, to=2-1]
	\arrow["l", from=2-1, to=1-1]
\end{tikzcd}\]
which we will denote by $\cat{SRGr}$, is endowed with the equations
\[
    s\circ l=\id_V\text{ and } t\circ l=\id_V.
\]
We abbreviate the category $\mathbf{Set}^{\cat{SRGr}}$ of reflexive graphs by $\cat{RGrph}$. We will see that $\cat{RGrph}$ admits a richer category of lassos than $\cat{Grph}$. Indeed, we obtain the category depicted in Figure~\ref{fig:rgrph_lassos}. Since its explication requires a few pages, we leave the details to Appendix~\ref{appendix:reflexive_graphs}.
}

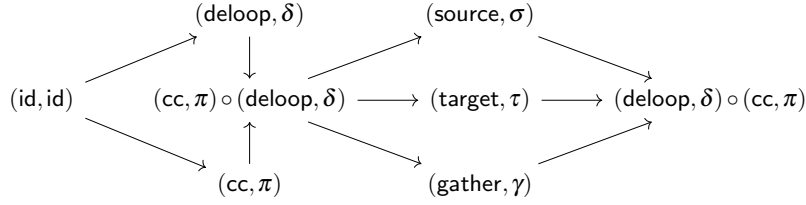
\begin{figure}
    \centering
    {\footnotesize
\[\begin{tikzcd}
	& {(\cat{deloop},\delta)} & {(\cat{source},\sigma)} \\
	{(\id,\id)} & {(\cat{cc},\pi)\circ(\cat{deloop},\delta)} & {(\cat{target},\tau)} & {(\cat{deloop},\delta)\circ(\cat{cc},\pi)} \\
	& {(\cat{cc},\pi)} & {(\cat{gather},\gamma)}
	\arrow[from=1-2, to=2-2]
	\arrow[from=1-3, to=2-4]
	\arrow[from=2-1, to=1-2]
	\arrow[from=2-1, to=3-2]
	\arrow[from=2-2, to=1-3]
	\arrow[from=2-2, to=2-3]
	\arrow[from=2-2, to=3-3]
	\arrow[from=2-3, to=2-4]
	\arrow[from=3-2, to=2-2]
	\arrow[from=3-3, to=2-4]
\end{tikzcd}\]
    }
    \caption{The category of lassos on $\cat{RGrph}$}
    \label{fig:rgrph_lassos}
\end{figure}

\begin{remark}
    It is also interesting to mention that, unlike the other lassos we've seen so far, some of the lassos in Figure~\ref{fig:rgrph_lassos} are not idempotent. In particular, there are some lassos for which performing the action of the lasso twice is equivalent to the action of the terminal lasso, and not the original lasso.
\end{remark}

\subsubsection{Edge-coloured graphs}

\body{
Whereas $\cat{Grph}$ and $\cat{RGrph}$ are defined by the schema $\cat{SGr}$ and $\cat{SRGr}$, respectively, we can represent $k$-edge-coloured graphs using following schema. Denote by $\cat{CGr}_k$ the following category.
{\footnotesize
\[\begin{tikzcd}
	{E_1} & \ldots & {E_k} \\
	& V
	\arrow["{s_1}"', shift right=2, from=1-1, to=2-2]
	\arrow["{t_1}", shift left, from=1-1, to=2-2]
	\arrow["{s_k}"', shift right, from=1-3, to=2-2]
	\arrow["{t_k}", shift left=2, from=1-3, to=2-2]
\end{tikzcd}\]
}
Then the objects of $\mathbf{Set}^{\cat{CGr}_k}$ can be interpreted as graphs whose edges recieve one of $k$ colours. The morphisms between them are exactly the graph homomorphisms which preserve the colour of edges.
}

\body{
For each colour $i=1,\ldots,k$, we obtain a lasso $(\La_i,\eta^i)$ which acts on $k$-edge-coloured graphs by identifying vertices which are connected by edges of colour $i$. Given lassos $(\La_i,\eta_i)$ and $(\La_j,\eta_j)$, we obtain a lasso $(\La_{i,j},\eta_{i,j})$ by composing the underlying maps. We remark that it doesn't matter the order in which we do this. The category of lassos on $\mathbf{Set}^{\cat{CGr}_k}$ is generated by composing the basic lassos $(\La_i,\eta_i)$. That is, it is the directed $k$-cube.
}

\begin{example}
    The category of lassos on $\mathbf{Set}^{\cat{CGr}_3}$, the category of $3$-edge-coloured graphs, is the following cube.

{\footnotesize
\[\begin{tikzcd}
	& {(\id,\id)} \\
	{(\La_1,\eta^1)} & {(\La_2,\eta^2)} & {(\La_3,\eta^3)} \\
	{(\La_{1,2},\eta^{1,2})} & {(\La_{1,3},\eta^{1,3})} & {(\La_{2,3},\eta^{2,3})} \\
	& {(\La_{1,2,3},\eta^{1,2,3})=(\cat{cc},\pi)}
	\arrow[from=1-2, to=2-1]
	\arrow[from=1-2, to=2-2]
	\arrow[from=1-2, to=2-3]
	\arrow[from=2-1, to=3-1]
	\arrow[from=2-1, to=3-2]
	\arrow[from=2-2, to=3-1]
	\arrow[from=2-2, to=3-3]
	\arrow[from=2-3, to=3-2]
	\arrow[from=2-3, to=3-3]
	\arrow[from=3-1, to=4-2]
	\arrow[from=3-2, to=4-2]
	\arrow[from=3-3, to=4-2]
\end{tikzcd}\]
}

\end{example}

\subsubsection{simplicial sets}
\body{
The category of \define{simplicial sets} extends the category of graphs. We define it to be $\mathbf{Set}^{\Delta^{\op}}$, the category of presheaves over the simplex category $\Delta$ (c.f. \cite[Ex. 1.3.7. pt. vi]{riehl2017category}).
}

\body{
Given $n\ge 1$ and a simplicial set $D$, let $\sim_n$ be the relation on vertices of $D$ given by being contained in the same $n$-dimensional face and let $\approx_n$ be its transitive closure. For a graph $G$, viewed as a simplicial set, the relation $\approx_1$ is simply the connectivity relation on vertices.
}

\body{
For each $n\ge 1$, the relation $\approx_n$ defines a lasso $(\cat{cc}_n,\eta^n)$. Indeed, its action identifies all related vertices. For each pair $i,j$, the action of $\cat{cc}_i$ commutes with the action of $\cat{cc}_j$. Furthermore, lassos with a lower index are strictly stronger than those with a higher index, in the sense that the composition of a pair of such lassos is equal to the one with a lower index. This is shown below. However, we refrain from providing an exhaustive characterization of the category of lassos of simplicial sets (like we have done for $\cat{Grph}$ and $\cat{RGrph}$). Instead this is left as future work.
\[\begin{tikzcd}
	{(\id,\id)} & \ldots & {(\cat{cc}_3,\eta^3)} & {(\cat{cc}_2,\eta^2)} & {(\cat{cc}_1,\eta^1)}
	\arrow[from=1-1, to=1-2]
	\arrow[from=1-2, to=1-3]
	\arrow[from=1-3, to=1-4]
	\arrow[from=1-4, to=1-5]
\end{tikzcd}\]
}

\section{Lifting Contractions to Diagrams}
\label{sec:monotone}
\body{
In this section we will show how to push decompositions forward along lasso contractions. The picture to keep in mind is the following. 
\begin{center}
    \includegraphics[width=.4\textwidth]{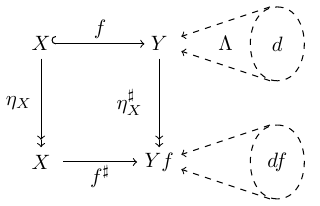}
\end{center}
Here we are given a lasso contraction $Y \twoheadrightarrow Y\contract{f}$ along a subobject $X \hookrightarrow Y$. The goal is to be able to push any structured decomposition $d$ whose colimit is the object $Y$ forward along the contraction map $\eta_X^\#$ (by image factorizations) to obtain a structured decomposition $d\contract{f}$ whose colimit is $Y\contract{f}$.
}

\body{
    It turns out to be slicker to prove something slightly different to our main result. Indeed, we make the following definition of `strong lassos' and work with them for the rest of the section. The proof of our main result can easily be derived from the proofs of the following (as we will explain at the end of this section).
}

\begin{definition}[Strong lasso]\label{def:strong-lasso}
    Let $(\La,\eta)$ be a lasso on a category $\cat{C}$. We say that $(\La,\eta)$ is \define{strong} if the functor $\La$ preserves the colimits of all monic diagrams (on top of preserving monic pushouts).
\end{definition}

\body{\label{para:cant_just_contract}
    One might think it would be possible to build our required diagram simply by contracting each individual object in $d$ along its `intersection' with $X$. However the following example for graphs shows that this is not the case.
    \begin{center}
    \includegraphics[width=.9\textwidth]{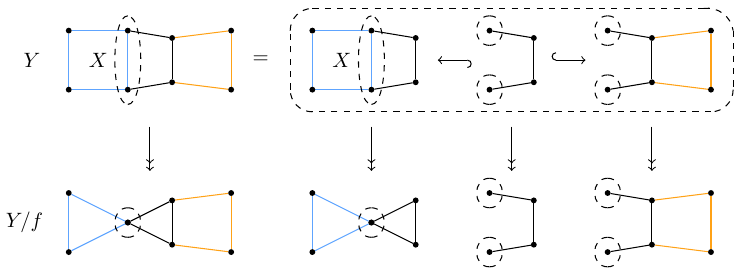}
\end{center}
    Indeed, the diagram of three graphs in the top right builds $Y$ as a colimit, but we cannot assemble the three graphs in the bottom right into a tree decomposition which builds $Y/f$ (indeed note that one no longer has a span of monomorphism). The graphs in the bottom right are obtained by contracting the subgraphs corresponding to $X$ in each of the three top-right graphs. All of this is to say that it is not enough to simply apply the lasso contraction point-wise. Instead we will make use of image factorizations of the global lasso contraction.
}

\body{
We now seek to introduce the notion of a \textit{diagram of images} (Definition~\ref{def:diagram-of-images}). To define it formally, we will first need to establish Lemma~\ref{lemma:images_map} below. In the meantime, as a first pass, we remark that a diagram of images is exactly what it sounds like: it is the diagram $\imageObj(\Omega, d)$ one obtains by point-wise image factorizations of a cocone $\Omega$ sitting above a given diagram $d\colon\cat{J}\to\cat{C}$. This is visualized as follows. 
\begin{center}
{\footnotesize
    \includegraphics[width=.7\textwidth]{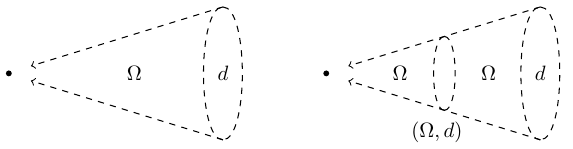}
    }
\end{center}

}\label{para:diagram-of-images-figure}

\begin{lemma}\label{lemma:images_map}
Consider the following commutative diagram in a category $\cat{C}$ which admits all image factorizations.
\[\begin{tikzcd}[row sep=small]
	&& B \\
	A \\
	& G &&& H
	\arrow["b"'{pos=0.3}, from=1-3, to=3-2]
	\arrow["q", from=2-1, to=1-3]
	\arrow["a"', from=2-1, to=3-2]
	\arrow["f"', from=3-2, to=3-5]
\end{tikzcd}\]
If $bq = a$, then there exists a unique monomorphism $u \colon \image fa\rightarrow\image fb$ that makes the following diagram commute. 
\[\begin{tikzcd}[row sep=small]
	&& B &&& {\imageObj (fb)} \\
	A &&& {\imageObj (fa)} \\
	& G &&& H
	\arrow["{\mi (fb)}", color={rgb,255:red,117;green,36;blue,112}, from=1-3, to=1-6]
	\arrow["b"'{pos=0.3}, from=1-3, to=3-2]
	\arrow["{\image (fb)}", color={rgb,255:red,117;green,36;blue,112}, hook, from=1-6, to=3-5]
	\arrow["q", from=2-1, to=1-3]
	\arrow["{\mi (fa)}"'{pos=0.7}, color={rgb,255:red,117;green,36;blue,112}, from=2-1, to=2-4]
	\arrow["a"', from=2-1, to=3-2]
	\arrow["{\exists !  u}", color={rgb,255:red,214;green,92;blue,92}, dashed, hook, from=2-4, to=1-6]
	\arrow["{\image (fa)}", color={rgb,255:red,117;green,36;blue,112}, hook, from=2-4, to=3-5]
	\arrow["f"', from=3-2, to=3-5]
\end{tikzcd}\]
\end{lemma}
\begin{proof}
Consider the path of morphisms $A \xrightarrow{q} B \xrightarrow{b} G \xrightarrow{f} H$. Taking two image factorizations and applying the universal property of $\imageObj(fbq)$, one finds a unique morphism $u \colon \imageObj(fbq) \to \imageObj(fb)$ making the following commute.
\[\begin{tikzcd}[row sep=small]
	A & B && G && H \\
	&&& {\imageObj (fb)} \\
	&&& {\imageObj(fbq)}
	\arrow["q", from=1-1, to=1-2]
	\arrow[curve={height=30pt}, from=1-1, to=3-4]
	\arrow["b"{pos=0.3}, hook, from=1-2, to=1-4]
	\arrow["{\mi (fb)}"', from=1-2, to=2-4]
	\arrow["f", from=1-4, to=1-6]
	\arrow["{\image (fb)}"', hook, from=2-4, to=1-6]
	\arrow[curve={height=30pt}, hook, from=3-4, to=1-6]
	\arrow["{\exists ! u}", dashed, from=3-4, to=2-4]
\end{tikzcd}\]
This concludes the proof since $fbq = fa$ and since $u$ is monic (q.v. Paragraph~\ref{para:mono-triangle}).
\end{proof}

\begin{definition}\label{def:diagram-of-images}
    Given a diagram $d \colon \cat{J} \to \cat{C}$ and a cocone $\Omega$ over $d$, we can define a diagram $\imageObj(\Omega, d) \colon\cat{J} \to \cat{C}$ whose objects come from taking the image factorisation of each $\Omega_i$ and whose maps arise as in Lemma~\ref{lemma:images_map}. We call the diagram obtained in this way the \define{diagram of images obtained from the cocone $\Omega$} (cf. Paragraph~\ref{para:diagram-of-images-figure}).
\end{definition}

\body{
Using the language of Definition~\ref{def:diagram-of-images}, we can now state Theorem~\ref{thm:decomp-contraction} which establishes the desired diagram corresponding to a lasso contraction.  
}

\begin{theorem}\label{thm:decomp-contraction}
    Let $(\La,\eta)$ be a strong lasso on a mono-strong category $\cat{C}$. Let Y be an object of $\cat{C}$ which arises as a colimit of a monic diagram $d:\cat{J} \to \cat{C}$. Denoting by~$\Lambda$ the cocone sitting above $d$, for any monomorphism $f \colon X \hookrightarrow Y$, let $d\contract{f} \colon \cat{J} \rightarrow \cat{C}$ denote the diagram of images $\imageObj(\eta_X^\sharp \Lambda, d) \colon\cat{J} \to \cat{C}$ obtained from the cocone $\eta_X^\sharp \Lambda$. That is, the cocone whose components are the composition of the components of $\Lambda$ with the contraction map~$\eta_X^\sharp$. Then $d\contract{f}$ has colimit $Y\contract{f}$.
\end{theorem}

\body{Rather than giving a proof of Theorem~\ref{thm:decomp-contraction} (we defer it to later on in this section, q.v. Paragraph~\ref{para:proof-of-thm:decomp-contraction}), we will instead construct the required diagram $d\contract{f}$ as a pushout of diagrams (Proposition~\ref{prop:local-contractions-give-correct-colimit} and Theorem~\ref{thm:local-contractions-give-correct-colimit}) and then prove this construction is equivalent to that of Theorem~\ref{thm:decomp-contraction}. This indirect proof approach of ours has a two benefits: (1) it clarifies the cryptomorphic relationship between the two constructions and (2) it allows for a short and conceptual proof (see Proposition~\ref{prop:local-contractions-give-correct-colimit}).
}

\body{\label{para:construction-subobject}
Building towards this more conceptual proof, consider now a lasso contraction $Y\contract{f}$ of an object $Y$ with respect to a subobject $f \colon X \hookrightarrow Y$. Given any $\cat{J}$-shaped diagram $d$ with colimit $(Y, \Lambda)$, one can define a family of objects $X_i := X\times_Y d(i)$ indexed by  $i\in\cat{J}$ by point-wise pullback of $f$ and $\Lambda_i$. 
\[\begin{tikzcd}[row sep = small]
	& {X_i=X\times_Y d(i)} \\
	X & Y & {d(i)} \\
	{\La X} & {Y\contract{f}}
	\arrow["{\lambda_i}"', hook, from=1-2, to=2-1]
	\arrow["\lrcorner"{anchor=center, pos=0.125, rotate=-45}, draw=none, from=1-2, to=2-2]
	\arrow["{f_i}", hook, from=1-2, to=2-3]
	\arrow["f", hook, from=2-1, to=2-2]
	\arrow["{\eta_X}", two heads, from=2-1, to=3-1]
	\arrow["{\eta_X^\sharp}", two heads, from=2-2, to=3-2]
	\arrow["{\Lambda_i}"', hook, from=2-3, to=2-2]
	\arrow["{f^\sharp}", from=3-1, to=3-2]
	\arrow["\lrcorner"{anchor=center, pos=0.125, rotate=180}, draw=none, from=3-2, to=2-1]
\end{tikzcd}\]
When $\cat{C}$ is mono-strong, the family $X_i$ assembles into a $\cat{J}$-shaped diagram whose colimit object is $X$. Moreover, this can be done functorially: the following  result, whose proof we defer to the apendix, is a corollary of theorem of  Bumpus, Kocsis,  Master and Minichiello~\cite{structured-decompositions}.
}

\begin{corollary}\label{cor:submon_extension}
    The functor $\mathsf{MDgm}:\cat{C}^{\op} \to\mathbf{Set}$ in Lemma~\ref{lemma: monic diagram presheaf}, which sends objects to their sets of structured decompositions, extends to a functor \[\Ma:\cat{C}^{\op}\to \mathbf{Cat}.\]
\end{corollary}
\begin{proof}
    See Appendix~\ref{appendix:laxity}.
\end{proof}

\body{\label{para:defining-Q}
We have already seen in Paragraph~\ref{para:cant_just_contract} that we cannot naively construct our new diagram by individually contracting each $d(i)$ along the subobjects $f_i \colon X_i \hookrightarrow d(i)$ as defined in Paragraph~\ref{para:construction-subobject}. Instead, we want each object to respect the global changes that happen in the contraction along $f$ that might not be respected locally in a contraction along each $f_i$. To that end, we define a diagram whose objects are the images of each $X_i$ under the whole contraction map. The first step is to show that any diagram has the same colimit as its associated diagram of images. 
}

\begin{lemma}\label{lemma:colimit-preseved-by-images}
    Let $\cat{C}$ be a category with image factorizations and colimits of shape~$\cat{J}$. If $d:\cat{J}\rightarrow\cat{C}$ is a diagram with colimit cocone $\Lambda$, then $\colim \imageObj(\Lambda, d) \cong \colim d$. 
\end{lemma}
\begin{proof}
We will show that the category of cocones over $d$ is isomorphic to the category of cocones over $\imageObj(\Lambda, d)$. As such, the two categories will have the same terminal object: i.e. $\colim d \cong \colim \imageObj(\Lambda, d)$. 

To see that, suppose that $(z_i)_{i \in J}$ is a family of cocone arrows over $d$ with apex $Z$. Let $f \colon i \to j$ be a morphism in $\cat{J}$ and consider the following diagram. 
\[\begin{tikzcd}[row sep=small]
	{d(i)} && {\imageObj(\Lambda, d)(i)} \\
	&&& {\colim d} & Z \\
	{d(j)} && {\imageObj(\Lambda, d)(j)}
	\arrow["{\mi \Lambda_i}"{description}, two heads, from=1-1, to=1-3]
	\arrow["{z_i}"{description}, curve={height=-30pt}, from=1-1, to=2-5]
	\arrow["{d(f) }", from=1-1, to=3-1]
	\arrow["{\im \Lambda_i}"{description}, hook, from=1-3, to=2-4]
	\arrow["{\imageObj(\Lambda, d)(f)}", from=1-3, to=3-3]
	\arrow["w"{description}, from=2-4, to=2-5]
	\arrow["{z_j}"{description}, curve={height=30pt}, from=3-1, to=2-5]
	\arrow["{\mi \Lambda_j}"{description}, two heads, from=3-1, to=3-3]
	\arrow["{\im \Lambda_j}"{description}, hook, from=3-3, to=2-4]
\end{tikzcd}\]
Since $((z_i)_{i \in J}, Z)$ is a cocone, we have that $z_i = z_j \circ d(f)$ and
so \[w \circ \im(\Lambda_i) \circ \mi(\Lambda_i) = z_i = z_j\circ d(f) =  w \circ \im(\Lambda_j) \circ \imageObj(\Lambda, d)(f) \circ \mi(\Lambda_i)\] and hence, since $\mi(\Lambda_i)$ is an epimorphism, we have that \[w \circ \im(\Lambda_i) = w \circ \im(\Lambda_j) \circ \imageObj(\Lambda, d)(f).\]
In other words, we have that $((w \circ \im(\Lambda_i))_{i \in \cat{J}} , Z)$ is a cocone over $\imageObj(\Lambda, d)$. 

Conversely, suppose that $(z_i)_{i \in J}$ is a family of cocone arrows over $\imageObj(\Lambda, d)$ with apex $Z$ as shown in the following diagram. 
\[\begin{tikzcd}[row sep=small]
	{d(i)} && {\imageObj(\Lambda, d)(i)} \\
	&&& {\colim d} & Z \\
	{d(j)} && {\imageObj(\Lambda, d)(j)}
	\arrow["{\mi \Lambda_i}"{description}, two heads, from=1-1, to=1-3]
	\arrow["{d(f) }", from=1-1, to=3-1]
	\arrow["{\im \Lambda_i}"{description}, hook, from=1-3, to=2-4]
	\arrow["{z_i}"{description}, curve={height=-30pt}, from=1-3, to=2-5]
	\arrow["{\imageObj(\Lambda, d)(f)}", from=1-3, to=3-3]
	\arrow["w"{description}, from=2-4, to=2-5]
	\arrow["{\mi \Lambda_j}"{description}, two heads, from=3-1, to=3-3]
	\arrow["{\im \Lambda_j}"{description}, hook, from=3-3, to=2-4]
	\arrow["{z_j}"{description}, curve={height=30pt}, from=3-3, to=2-5]
\end{tikzcd}\]
Since $\imageObj(\Lambda, d)(f) \circ \mi(\Lambda_i) = \mi(\Lambda_j) \circ d(f)$, one has that $((z_i \circ \mi(\Lambda_i)_{i \in \cat{J}}, Z)$ is a cocone over $d$, as desired. 

Thus we have shown that every cocone over $d$ yields a cocone over $\imageObj(\Lambda, d)$ and vice-versa. This can be easily seen to yield an isomorphism of categories $\mathsf{Cocones}(d) \cong \mathsf{Cocones}(\imageObj(\Lambda, d))$.
\end{proof}

\begin{lemma}\label{def_Q}
    Let $(\La, \eta)$ be a strong lasso on an mono-strong category $\cat{C}$ with all colimits. Given any monic diagram $d$ with colimit cocone $(\Lambda,X)$, let 
    \[\Qa_X \colon d \mapsto \imageObj(\eta_X\Lambda, d)\] be the function taking $d$ to the diagram of images obtained from $d$ under the cocone $\eta_X\Lambda$. Then for each such $d$, the colimit of $\Qa_X(d)$ is $\La X$.
\end{lemma}
\begin{proof}
By Lemma~\ref{image_epi_precomp} and the fact that $\eta_X\circ\Lambda_i=\La\Lambda_i\circ\eta_{X_i}$ we have $\Qa_X(d)(i)\cong \imageObj{\La\Lambda_i}$.
Now we may apply Lemma~\ref{lemma:colimit-preseved-by-images} to the diagrams $\Qa_X(d)$ and $\La X_i$, and this gives us that that $\colim \Qa_X(d) \cong \colim_i \La X_i$. Since $\La$ preserves monic pushouts, we have that $\colim_i\La X_i=\La X$, as desired.
\end{proof}

\begin{lemma}\label{lemma:definition-of-Q}
The functions in Lemma~\ref{def_Q} assemble into a \textit{lax} natural transformation
\[\begin{tikzcd}
	{\cat{C}^{\op}} && {\cat{Cat}}
	\arrow[""{name=0, anchor=center, inner sep=0}, "\Ma", curve={height=-12pt}, from=1-1, to=1-3]
	\arrow[""{name=1, anchor=center, inner sep=0}, "{\Ma \circ \La^{\op}}"', curve={height=12pt}, from=1-1, to=1-3]
	\arrow["\Qa", shorten <=3pt, shorten >=3pt, Rightarrow, from=0, to=1]
\end{tikzcd}\]
\end{lemma}
\begin{proof}
    See Appendix~\ref{appendix:laxity}.
\end{proof}

\body{
For any morphism $f \colon X \to Y$, Lemma~\ref{lemma: monic diagram presheaf} and Lemma~\ref{lemma:definition-of-Q} establish the following pair of composable functors: $\Ma \La X \xleftarrow{\Qa_X} \Ma X \xleftarrow{\Ma f} \Ma Y$. To recall and clarify how these functors operate, it is useful to consider their action on any fixed $d \in \Ma Y$ (i.e. for any diagram $d \colon \cat{J} \to \cat{C}$ with colimit $Y$). To that end, consider the following diagram which shows how one constructs $\Ma(f)(d)$ (resp. $(\Qa_X \circ\Ma(f))(d)$) by pointwise pullbacks (resp. pointwise image factorizations). 
\[
\includegraphics[width=.7\textwidth]{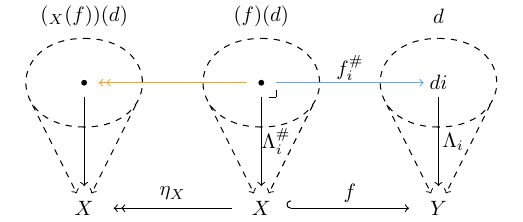}
\]
These constructions yield two families of morphisms. The first is the family \[\Bigl(\Ma(f)(d)(j) {\color{diagramBlue}\xrightarrow{f^{\#}_i}} d(j)\Bigr)_{j \in \cat{J}},\] while the second is the family \[\Bigl(\Ma(f)(d)(j) {\color{diagramGold}\xrightarrow{\mi(\eta_X \Lambda_i^\#)}} (\Qa_X \circ\Ma(f))(d)(j)\Bigr)_{j \in \cat{J}}.\] Viewed as morphisms in $\cat{C}^\cat{J}$, these assemble into the following span:
\begin{equation}\label{eqn:span-of-diagrams}
    (\Qa_X \circ\Ma(f))(d) {\color{diagramGold}\longleftarrow} \Ma(f)(d) {\color{diagramBlue}\longrightarrow} d.
\end{equation}
The following result shows that the pushout of this span is precisely the diagram we have been seeking to construct. 
}


\begin{proposition}\label{prop:local-contractions-give-correct-colimit}
    Let $(\La, \eta)$ be a strong lasso on a mono-strong category $\cat{C}$ with all colimits. Suppose $f \colon X \hookrightarrow Y$ is a monomorphism in $\cat{C}$ and that $d \in \Ma Y$ is a monic diagram of shape $\cat{J}$ whose colimit is $Y$. Then, working in $\cat{C}^\cat{J}$, where pushouts are computed pointwise, the pushout of the span of Equation~\eqref{eqn:span-of-diagrams} is a diagram with colimit $Y \contract{f}$.
\end{proposition}
\begin{proof}
Colimits commute with each other, thus we have that:
\begin{align*}
    \colim &\bigl ((\Qa_X \circ\Ma(f))(d) +_{\Ma(f)(d)} d \bigr ) = \\
    &= \colim \bigl((\Qa_X \circ\Ma(f))(d)\bigr) +_{\colim \Ma(f)(d)} \colim d  \\
    &= \La X +_{\colim \Ma(f)(d)} \colim d &\text{(Lemma~\ref{lemma:definition-of-Q})}\\
    &= \La X +_{X} \colim d &\text{(Lemma~\ref{lemma: monic diagram presheaf})}\\
    &= \La X +_{X} Y &\text{(definition of } d)\\
    &= Y \contract{f} &\text{(definition of } Y \contract{f}).
\end{align*}
\end{proof}

\body{
We can actually prove a stronger result than Proposition~\ref{prop:local-contractions-give-correct-colimit}. Indeed, we can forgo the assumption that the category $\cat{C}$ has all colimits.
For the proof of this result, we have to manually show that our construction satisfies the definition of a colimit.
}

\begin{theorem}[Upgrade of Proposition~\ref{prop:local-contractions-give-correct-colimit}]\label{thm:local-contractions-give-correct-colimit}
    Let $(\La, \eta)$ be a strong lasso on a mono-strong category $\cat{C}$. Suppose $f \colon X \hookrightarrow Y$ is a monomorphism in $\cat{C}$ and that $d \in \Ma Y$ is a monic diagram whose colimit is $Y$. Then the span Equation~\eqref{eqn:span-of-diagrams} has~$Y \contract{f}$ as its colimit.
\end{theorem}
\begin{proof}
    We abbreviate the diagram $\Ma(f)(d)$ by $x$ and the diagram $(\Qa_x\circ\Ma(f))(d)$ by $q$.
    Denote by $h\colon\cat{J}\to\cat{C}$ the diagram obtained via the following pointwise pushout.
\[\begin{tikzcd}[row sep=small]
	{x(i)} & {d(i)} \\
	{q(i)} & {q(i)+_{x(i)} d(i)=:h(i)}
	\arrow["{f_i}"', hook, from=1-1, to=1-2]
	\arrow["{\mi (\eta_X\Lambda_i)}"', two heads, from=1-1, to=2-1]
	\arrow[from=1-2, to=2-2]
	\arrow[from=2-1, to=2-2]
	\arrow["\ulcorner"{anchor=center, pos=0.125, rotate=180}, draw=none, from=2-2, to=1-1]
\end{tikzcd}\]
    By Lemmas~\ref{lemma: monic diagram presheaf} and \ref{para:defining-Q}, we have that the colimits of $d,x$ and $q$ are $Y,X$ and $\La X$, respectively. Observe that $Y\contract{f}$ forms the apex of a cocone over the pushout diagram defining~$h$. Hence there exists a unique map $\Omega_i\colon h(i)\to Y\contract{f}$ making the following commute.
    {\small
\[\begin{tikzcd}
	X &&& Y \\
	& {x(i)} & {d(i)} \\
	& {q(i)} & {h(i)} \\
	{\La X} &&& {Y\contract{f}}
	\arrow["f"{description}, hook, from=1-1, to=1-4]
	\arrow["{\eta_X}"{description}, two heads, from=1-1, to=4-1]
	\arrow[two heads, from=1-4, to=4-4]
	\arrow[hook, from=2-2, to=1-1]
	\arrow["{f_i}"', hook, from=2-2, to=2-3]
	\arrow["{\mi (\eta_X\Lambda_i)}"', two heads, from=2-2, to=3-2]
	\arrow[hook, from=2-3, to=1-4]
	\arrow["{u_i}", from=2-3, to=3-3]
	\arrow["{v_i}"', from=3-2, to=3-3]
	\arrow[from=3-2, to=4-1]
	\arrow["\lrcorner"{anchor=center, pos=0.125, rotate=180}, draw=none, from=3-3, to=2-2]
	\arrow["{\exists ! \Omega_i}", dashed, from=3-3, to=4-4]
	\arrow[hook, from=4-1, to=4-4]
\end{tikzcd}\]
}
    We are required to show that $(h,\Omega)$ form a colimit cocone with apex $Y\contract{f}$. To this end, suppose that $\Sigma$ is a cocone over $h$ with some apex $Z$. Observe that by composing $\Sigma$ with the collections of maps $(u_i)$ and $(v_i)$ in turn, we obtain cocones over $q$ and $d$, both having apex $Z$. Since $Y$ and $\La X$ are the colimits of $d$ and $q$, respectively, we obtain unique maps $s$ and $t$ making the following commute.

{\small
\begin{equation}\begin{tikzcd}
	X &&& Y \\
	& {x(i)} & {d(i)} \\
	& {q(i)} & {h(i)} \\
	{\La X} &&& {Y\contract{f}} & Z
	\arrow["f"{description}, hook, from=1-1, to=1-4]
	\arrow["{\eta_X}"{description}, two heads, from=1-1, to=4-1]
	\arrow[two heads, from=1-4, to=4-4]
	\arrow["{\exists ! s}"{description}, curve={height=-12pt}, dashed, from=1-4, to=4-5]
	\arrow[hook, from=2-2, to=1-1]
	\arrow["{f_i}"', hook, from=2-2, to=2-3]
	\arrow["{\mi (\eta_X\Lambda_i)}"', two heads, from=2-2, to=3-2]
	\arrow[hook, from=2-3, to=1-4]
	\arrow["{u_i}", from=2-3, to=3-3]
	\arrow["{v_i}"', from=3-2, to=3-3]
	\arrow[from=3-2, to=4-1]
	\arrow["\lrcorner"{anchor=center, pos=0.125, rotate=180}, draw=none, from=3-3, to=2-2]
	\arrow["{\Omega_i}"{description}, from=3-3, to=4-4]
	\arrow["{g_i}", curve={height=-18pt}, from=3-3, to=4-5]
	\arrow[hook, from=4-1, to=4-4]
	\arrow["{\exists ! t}"{description}, curve={height=18pt}, dashed, from=4-1, to=4-5]
\end{tikzcd}\label{diag:final_main}
\end{equation}
}
Since $Y\contract{f}$ is a pushout of $\La X$ and $Y$ over $X$, we obtain a unique morphism from $Y\contract{f}$ to $Z$, commuting with the above, as required.
\end{proof}

\body{
Theorem~\ref{thm:local-contractions-give-correct-colimit} implies that one can start with any diagram $d$ with colimit $Y$ and, so long as $d$ is a diagram of monomorphisms, then one can produce a diagram whose colimit is the $(\La, \eta)$-contraction $Y\contract{f}$ of $Y$ along a monomorphism $f\colon X \hookrightarrow Y$. This is very close to being the result we have been working towards throughout this section; however, what it is missing is any guarantee that the diagram constructed in Theorem~\ref{thm:local-contractions-give-correct-colimit} is itself a diagram of monomorphisms. Fortunately, this can be fixed by taking pointwise image factorizations of its colimit cocone as the following corollary points out.  
}

\begin{corollary}\label{cor:main}
Let $\cat{C}$ be a mono-strong category with colimits. Suppose $f \colon X \hookrightarrow Y$ is a monomorphism in $\cat{C}$ and that $d \in \Ma Y$ is a monic diagram whose colimit is $Y$. Then the diagram of images obtained from the colimit cocone sitting above $(\Qa_X \circ\Ma(f))(d) +_{\Ma(f)(d)} d$ has $Y\contract{f}$ as its colimit. 
\end{corollary}
\begin{proof}
Combine Lemma~\ref{lemma:colimit-preseved-by-images} and Theorem~\ref{thm:local-contractions-give-correct-colimit}.
\end{proof}

\body{
Finally, we prove that our construction is equivalent to the simpler one given in Theorem~\ref{thm:decomp-contraction}.
}
\body{
\begin{proof}[Proof of Theorem~\ref{thm:decomp-contraction}]
    We will show that $d\contract{f}$ is identical to the diagram obtained in Corollary~\ref{cor:main}. Recall Diagram~\ref{diag:final_main}. Since pushouts preserve epimorphisms, we have that the maps $u_i\colon d(i)\to h(i)$ are epic. Hence by Lemma~\ref{image_epi_precomp} the images of $\Omega_i$ and $\Omega_iu_i=\eta_X^\sharp\Lambda_i$ coincide for each $i\in\cat{J}$. Since this is how both diagrams are defined, this completes our proof.
\end{proof}
}\label{para:proof-of-thm:decomp-contraction}

\body{\begin{proof}[Proof of Theorem~\ref{thm:main}]
    One can readily verify that the constructions in Lemma~\ref{lemma: monic diagram presheaf} and Lemma~\ref{lemma:definition-of-Q} yield diagrams of the same shape as the diagram $d$ one starts with. In particular, this means that all of the results above can be specialized to the case in which one only seeks to construct objects via diagrams whose shape is limited to a special class (for example, trees). Furthermore, observe that both Lemma~\ref{lemma: monic diagram presheaf} and Lemma~\ref{lemma:definition-of-Q} preserve diagrams of monomorphisms (i.e. they send a diagram $d$ whose arrows are all monic to a new diagram whose arrows are also all monic). These two facts imply that all of the results above can be specialized to the case in which $d$ is a structured decomposition of some given shape (for example, a tree decomposition of graphs).

    All of this implies that, if instead of considering diagrams of any shape one only cares about structured decompositions, then all the results of this section can actually be stated for lassos which are not strong (namely lassos which preserve pushouts of monomorphisms, but not colimits in general, cf. Definitions~\ref{def:lasso} and~\ref{def:strong-lasso}). 
\end{proof}}\label{para:proof-of-thm:main}

\begin{remark}
    As in Section~\ref{sec:cat_of_lassos}, we could investigate categories of strong lassos. By Lemma~\ref{prop:cc_coco} and the surrounding results, we see that all the lassos on $\cat{Grph}$ are strong. However in Appendix~\ref{appendix:reflexive_graphs} we show that this is not always the case. In particular, see Example~\ref{eg:delooping_almost_bad} for a strictly weak lasso.
\end{remark}

\begin{remark}
    In the statement of Theorem~\ref{thm:decomp-contraction}, we assume that the morphism $f\colon X\hookrightarrow Y$ is monic. However we don't use this fact anywhere in the proof. This suggests that we could define contractions differently, and allow for contractions along any morphism. However, it is easy to check that in $\cat{Grph}$ a contraction defined in this way is always equivalent to a contraction defined along a monomorphism. Indeed, this is because given any map $f\colon X\rightarrow Y$, the connected components of $X$ are exactly the preimages of the connected components of $Y$ which have nonempty preimage. Hence the contraction along $\image f$ is isomorphic to the contraction along $f$. It is perhaps an interesting question to ask whether this holds in general in mono-strong categories, or not.
\end{remark}

\section{Discussion \& Open Problems}\label{sec:outlook}
\body{
Here we introduce the notion of a \textit{lasso}. This machinery allows us to systematically identify classes of epimorphisms -- called \textit{lasso contractions} -- in a category $\cat{C}$ along which one can push decompositions in a meaningful way. In more detail, lasso contractions have the property that, given any epimorphism $f \colon X \twoheadrightarrow Y$ and any tree-shaped structured decomposition $d_X$ of $X$, if $f$ arises as a lasso contraction, then one can obtain a decomposition $d_Y$ such that: (1) $d_Y$ is a decomposition with the same shape as $d_X$ and (2) $d_Y$ is a decomposition of $Y$. 
}
\body{
Instantiating our results in the category $\cat{Grph}$ of directed multigraphs, the above results can be used to answer the following question: for which classes of surjective homomorphisms $f \colon H \twoheadrightarrow G$ can one always push a tree decomposition for $H$ along $f$ to obtain a tree decomposition for $G$ of the same shape? In $\cat{Grph}$, as it turns out, \textit{contraction maps} are the only class of epimorphisms along which one can push decompositions while leaving their shape unchanged. Thus, in this particular instance, our general machinery yields a canonicity result which characterizes contraction maps. 
}
\body{ 
Not only is our category theoretic approach ideal for proving canonicity results,\footnote{This is a very well-known strength of category theory: since all definitions are stated up to isomorphism and since constructions are done via universal properties, there is a sense in which all the results one obtains are canonical.} but it also enables out proof techniques to be adapted to wide range of mathematical objects all at once. For example our results apply in any mono-strong category such as directed multigraphs, hypergraphs, Petri nets, circular port graphs, half-edge graphs, databases, simplicial sets or indeed any (co)presheaf category. 
}
\body{
Interestingly, although we can prove the above canonicity result for graph contractions, if one moves away from the category $\cat{Grph}$ of directed multigraphs to other more embellished categories of categorical data, a much richer story begins to appear. For instance, already in the category of \textit{reflexive graphs} one finds that there are many more kinds of lassos, each one defining a different class of epimorphisms along which (tree) decompositions can be pushed forward. Moreover, as we continue to add attributes to the combinatorial object under study, we observe that the corresponding categories of lassos appear to become more and more involved. This suggests that there should be a relationship between the defining schema of the combinatorial data and the resulting category of lassos. Understanding this relationship would be particularly useful in studying lassos on much more complicated categories of copresheaves than the one studied here. Thus we state the following conjecture. 
}

\begin{conjecture}\label{conj:lasso-functoriality}
There is a contravariant functor $\cat{Lasso} \colon \cat{Cat} \to \cat{Cat}$ which takes any functor $F \colon \cat{S} \to \cat{T}$ between schemata $\cat{S}$ and $\cat{T}$ to a functor $\cat{Lasso}(\cat{Set}^\cat{T}) \to \cat{Lasso}(\cat{Set}^\cat{S})$ between the categories of lassos on the associated copresheaf categories.
\end{conjecture}

\body{
A further open problem, possibly a more challenging one, is to determine whether pushing decompositions along morphisms is a functorial operation. Slightly more formally, one would like an analogue of Theorem~\ref{lemma: monic diagram presheaf} (or Corollary~\ref{cor:submon_extension}) for contractions; namely a \textit{covariant} functor which assigns to each object of a category the set of all of its decompositions and which acts functorially on lasso contractions. Obviously there is significant impediment to such a result: we do not know whether, given a lasso $(\La, \eta)$ on a category $\cat{C}$, it is even possible to construct a category $\cat{C}_{(\La, \eta)}$ whose objects are those of $\cat{C}$ but whose morphisms are only those epimorphisms in $\cat{C}$ which arise as $(\La, \eta)$-contractions.
The technical hurdle one needs to overcome is that of proving that lasso-contractions can be composed. Indeed, although the special case of graphs is very easy\footnote{It is well known that, if $f \colon G \twoheadrightarrow H$ and $g \colon H \to J$ are graph contractions, then $gf \colon G \twoheadrightarrow J$ also is.}, the general question of determining whether the composite of two $(\La, \eta)$-contractions is itself a $(\La, \eta)$-contraction has resisted our efforts thus far. We state this as the following conjecture. 
}

\begin{conjecture}\label{conj:contraction-composition}
Let $(\La, \eta)$ be a lasso on a category $\cat{C}$. Consider any two monomorphisms $f_1 \colon X_1 \hookrightarrow Y$ and $f_2 \colon X_2 \to Y\contract{f}$ as in the following diagram. 
\[\begin{tikzcd}[sep=small]
	Y &&& {Y\contract{f}} &&& {(Y\contract{f_1})\contract{f_2}} \\
	& {X_1} & {\La(X_1)} && {X_2} & {\La(X_1)}
	\arrow["{{\color{diagramBlue}\eta^{\#}_{X_1}}}", two heads, from=1-1, to=1-4]
	\arrow["{{\color{diagramBlue}\eta^{\#}_{X_2}}}", two heads, from=1-4, to=1-7]
	\arrow["{f_1}", hook, from=2-2, to=1-1]
	\arrow["{\eta_{X_1}}"', two heads, from=2-2, to=2-3]
	\arrow[from=2-3, to=1-4]
	\arrow["{f_2}", hook, from=2-5, to=1-4]
	\arrow["{\eta_{X_2}}"', two heads, from=2-5, to=2-6]
	\arrow[from=2-6, to=1-7]
\end{tikzcd}\]
Then there is a subobject $f \colon X \hookrightarrow Y$ whose lasso contraction 
\[\begin{tikzcd}[sep=small]
	Y &&&&&& {Y\contract{f}} \\
	&& X && {\La(X_1)}
	\arrow["{{\color{diagramGold}\eta^{\#}_{X}}}", two heads, from=1-1, to=1-7]
	\arrow["f", hook, from=2-3, to=1-1]
	\arrow["{\eta_{X}}"', two heads, from=2-3, to=2-5]
	\arrow[from=2-5, to=1-7]
\end{tikzcd}\]
satisfies $Y\contract{f} = (Y\contract{f_1})\contract{f_2}$ and ${\color{diagramGold}\eta_X^\#} = {\color{diagramBlue}\eta_{X_2}^\#\eta_{X_1}^\#}$.
\end{conjecture}

\paragraph{Acknowledgments:} we wish to thank Kevin Carlson for feedback and suggestions on a preliminary version of this article.

\appendix

\section{The category of lassos on reflexive graphs}
\label{appendix:reflexive_graphs}

\body{
Many of the proofs for $\cat{Grph}$ adapt to the category $\cat{RGrph}$. Indeed, the connected components lasso $(\cat{cc},\pi)$ restricts to a lasso on $\cat{RGrph}$, which we will also denote by $(\cat{cc},\pi)$. However, not all of the proofs transfer. In particular, Lemma~\ref{lemma:oh_no} relies on the existence of graphs without loops and so it fails, along with its consequences, namely Lemma~\ref{lemma:cant_identify_edges} and Corollary~\ref{cor:must_identify_vertices}. We prove the following alternatives.
}

\begin{lemma}\label{lemma:parallel_edge_particle}
    Let $(\Fa,\alpha)$ be a lasso. Suppose that the components of $\alpha$ are vertex-trivial. If for some graph $G$, we have that $\alpha_G$ identifies some pair of parallel non-loop edges in $G$, then $\alpha$ identifies all pairs of parallel edges in all graphs (including pairs of parallel loops).
\end{lemma}
\begin{proof}
Suppose there is a graph $G$ containing a pair of parallel non-loop edges $e$ and $e'$ such that $e$ and $e'$ are identified under $\alpha_G\colon G\twoheadrightarrow\Fa(G)$. Denote by $\reflexiveparallel$ the graph with two vertices, a loop at each vertex and a pair of parallel edges spanning the vertices and denote by $\suprisereverseedge$ the graph obtained from $\reflexiveparallel$ by adding a new edge in the opposite direction to the parallel edges.

Observe that there exists a homomorphism $f\colon G\to\suprisereverseedge$ so that $f(e)$ and $f(e')$ are exactly the pair of parallel edges of $\suprisereverseedge$. Then by the naturality of $\alpha$, the fact that $\alpha$ is the identity on vertex sets and the fact that $\alpha$ is epic, we have that: \begin{enumerate*}
    \item $\Fa$ is the identity on discrete reflexive graphs; 
    \item $\Fa$ preserves the single edge graph (i.e. $\Fa(\edgewithloops) = \edgewithloops$); and 
    \item $\Fa(\suprisereverseedge)$ is the graph $\fulltwo$, having two vertices, with loops, and one edge in each direction between them. 
\end{enumerate*} Consider the fact that $\suprisereverseedge$ is the monic pushout of $\fulltwo$ and $\edgewithloops$. Since $\Fa$ preserves monic pushouts we get that $\Fa(\reflexiveparallel)$ is  isomorphic to $\edgewithloops$.

Let $H$ be an arbitrary graph and suppose that there is a pair $e$, $e'$ of parallel edges (possibly loops) in $H$. Consider the homomorphism $h\colon\reflexiveparallel\to H$ that sends one parallel edge to $e$ and the other to $e'$. Now by the naturality of $\alpha$ and our analysis above, we get that $\alpha_H$ identifies the edges $e$ and $e'$. Since $H$ and the choice of parallel edges were arbitrary, this completes our proof.
\end{proof}

\begin{proposition}\label{prop:smooth_is_a_functor}
    Let $\cat{smooth}$ be the map $\cat{RGrph}\to\cat{RGrph}$ that sends each graph $G=(V,E)$ to a graph $\cat{smooth}(G)$ with vertex set $V$ and edge set obtained from~$E$ by identifying parallel edges (including parallel loops). Then the map $\cat{smooth}$ extends uniquely to a functor $\cat{smooth}\colon\cat{RGrph}\to\cat{RGrph}$ such that there is a natural transformation $\id_{\cat{RGrph}}\Rightarrow\cat{smooth}$.
\end{proposition}
\begin{proof}
    Let $f\colon G\to H$ be a graph homomorphism. We first claim that there is a unique homomorphism $\cat{smooth}(f)$ from $\cat{smooth}(G)$ to $\cat{smooth}(H)$ that acts identically to $f$ on vertices. Indeed, consider the reflexive closure of the relation of edges being parallel. This relation is preserved by graph homomorphism and so such a homomorphism as $\cat{smooth}(f)$ exists. It is unique since $\cat{smooth}(G)$ and $\cat{smooth}(H)$ have no pairs of parallel edges. Hence it remains to show that any functor extending the map $\cat{smooth}$ is the identity on the vertex-map of homomorphisms.
    
    Let $\eta$ be a natural transformation from $\id_\cat{RGrph}$ to $\cat{smooth}$. Clearly $\eta$ must be vertex-trivial. Hence by naturality we have our required property.
\end{proof}

\body{
    We call the functor defined by the above proposition the \define{smoothing functor}.
}

\begin{lemma}\label{lemma:smoothing}
The smoothing functor does not preserve monic pushouts in $\cat{RGrph}$.
\end{lemma}
\begin{proof}
    One finds a counterexample by considering the following monic pushout
    \begin{center}
        \includegraphics[width=.6\textwidth]{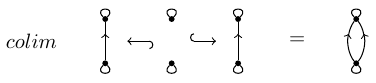}
    \end{center}
    which is not preserved under the action of the smoothing functor.
\end{proof}

\begin{lemma}\label{lemma:all_by_canon}
    Let $\Fa$ be a functor $\cat{RGrph}$ to $\cat{RGrph}$ along with a natural transformation $\alpha:\cat{id}_{\cat{RGrph}}\Rightarrow \Fa$ whose components are epimorphisms. If some component of $\alpha$ is not edge-trivial, then for each reflexive graph $G$, the map $\alpha_G$ identifies every class of parallel loops in $G$.
\end{lemma}
\begin{proof}
    Analogously to Lemma~\ref{lemma:condition_for_double_loop_particle}, one finds that $\La$ must map the graph $\twoloop$, with one vertex and two loops, to the terminal graph. Now consider an arbitrary reflexive graph~$G$. Suppose that $G$ contains a vertex $x$ with two loops $e$ and $e'$. Then there is a morphism $f\colon \twoloop\to G$ mapping the unique vertex of $\twoloop$ to $x$ and sending the two loops of $\twoloop$ to each of $e$ and $e'$. By naturality, we must have that $e$ and $e'$ are identified by $\eta_G$. This completes the proof.
\end{proof}

\begin{proposition}
    Let $\cat{deloop}$ be the map $\cat{RGrph}\to\cat{RGrph}$ that sends each graph $G=(V,E)$ to a graph $\cat{deloop}(G)$ with vertex set $V$ and edge set obtained from~$E$ by identifying loops which have the same endvertex. Then the map $\cat{deloop}$ extends uniquely to a functor $\cat{deloop}\colon\cat{RGrph}\to\cat{RGrph}$ such that there is a natural transformation $\id_{\cat{RGrph}}\Rightarrow\cat{deloop}$.
\end{proposition}
\begin{proof}
    By similar arguments to the proof of Proposition~\ref{prop:smooth_is_a_functor}, we obtain our functor. This functor sends each graph homomorphism $f\colon G\to H$ to a homomorphism from $\cat{deloop}(G)$ to $\cat{deloop}(H)$ that acts identically to $f$ on vertex sets, and on edges in the obvious way.
\end{proof}

\body{
    We call the functor defined in the above proposition the \define{delooping functor}.
}

\begin{lemma}\label{lemma:delooping_good}
    The delooping functor preserves monic pushouts in $\cat{RGrph}$.
\end{lemma}
\begin{proof}
Consider an arbitrary monic span $A\hookleftarrow C\hookrightarrow B$ in $\cat{RGrph}$ and denote by $G$ its pushout. Since $\delta$ is vertex-trivial, it is sufficient to prove that, if none of the graphs $A,B$ or $C$ possess parallel loops, then their pushout $G$ does not. So let $u$ be a vertex in $G$. Suppose first that $u$ is not in the image of $C$. Then $u$ has exactly one loop, since its preimage (which is a singleton since the morphisms are monic) in either $A$ or $B$ also has exactly one loop. Instead suppose that $u$ is in the image of~$C$. Then there is a loop at each preimage of $u$ in both $A$ and $B$. However, both of these loops are the image of a single loop in $C$. In this case, $u$ also has exactly one loop. This completes our proof.
\end{proof}

\body{
    By the above analysis, the delooping functor gives us a lasso on $\cat{RGrph}$ which we denote by $(\cat{deloop},\delta)$.
}

\begin{example}\label{eg:delooping_almost_bad}
Although the delooping functor preserves monic pushouts, it does not preserve all monic colimits. In particular, it doesn't preserve all coequalisers. Indeed, consider the diagram $a,b\colon\terminalgraph\hookrightarrow \edgewithloops$ where $\terminalgraph$ is the terminal graph, $\edgewithloops$ is the reflexive graph on two vertices with a single edge between them and $a$ and $b$ are the two possible homomorphisms $\terminalgraph\hookrightarrow\edgewithloops$. Now observe that the colimit of this diagram is the graph with one vertex and three distinct loops.
\end{example}

\begin{lemma}
    Let $(\La,\eta)$ be a nontrivial lasso on $\cat{RGrph}$ and suppose that it is vertex-trivial. Then $(\La,\eta)$ is the delooping lasso $(\cat{deloop},\delta)$.
\end{lemma}
\begin{proof}
    Since $(\La,\eta)$ is nontrivial, but vertex-trivial, there is a graph $G$ so that $\eta_G$ identifies some pair of parallel edges in $G$. If such a pair of edges are not loops, then by Lemma~\ref{lemma:parallel_edge_particle} the functor $\La$ has to be the smoothing functor. However, by Lemma~\ref{lemma:smoothing}, this can't be the case. Hence, no pair of non-loop parallel edges may be identified. So the identified edges must be a pair of parallel loops and by Lemma~\ref{lemma:all_by_canon}, we conclude that $\La$ is the delooping functor.
\end{proof}

\body{
One may prove an analogue of Lemma~\ref{lemma:basically_cc} for $\cat{RGrph}$ (q.v.\ (\ref{cor:like_cc})), which states that any non-vertex-trivial lasso on $\cat{RGrph}$ must act like the connected components lasso on vertices. We formalise this in the following definition. It then remains to deduce the possibilities for how such a lasso can act on edges.
}

\begin{definition}
    We say that a lasso $(\La,\eta)$ on $\cat{RGrph}$ is \define{component-collapsing} if, for all graphs $G$ in $\cat{RGrph}$, the graph $\La(G)$ has the same vertex set as the graph~$\cat{cc}(G)$ and the vertex map of the homomorphism $\eta_G\colon G\twoheadrightarrow\La(G)$ is identical to the vertex map of the homomorphism $\pi_G\colon G\twoheadrightarrow \cat{cc}(G)$. That is, the lasso acts like the connected components lasso $(\cat{cc},\pi)$ on the vertex sets of graphs. Observe that~$(\cat{cc},\pi)$ is initial in the subcategory of component-collapsing lassos.
\end{definition}

\begin{corollary}\label{cor:like_cc}
    Let $(\La,\eta)$ be a lasso on $\cat{RGrph}$. If $(\La,\eta)$ is not vertex-trivial, then it is component-collapsing.
\end{corollary}
\begin{proof}
    The proof is analogous to the proof of Lemma~\ref{lemma:basically_cc}.
\end{proof}

\begin{proposition}\label{prop:source}
    There is a component-collapsing lasso $(\cat{source},\sigma)$ so that for each graph $G$, the map $\sigma_G$ identifies exactly the edges having $u$ as a source, for each vertex $u$ in $G$.
\end{proposition}
\begin{proof}
    For each reflexive graph $G$, let $\cat{source}(G)$ be the graph obtained from $\cat{cc}(G)$ by identifying edges sharing a source in $G$. It is easy to see that this assembles into a functor and that the maps $\sigma_G$ form a natural transformation $\sigma\colon\id_\cat{RGrph}\Rightarrow\cat{source}$.
    Now consider an arbitrary monic span $A\hookleftarrow C\hookrightarrow B$ in $\cat{RGrph}$ and denote by $G$ its pushout. We are required to show that quotienting edges by the relation of sharing a source commutes with monic pushouts, which in turn consist of a disjoint union and then a quotient by the relation on edges given by the span. Let $a$ and $b$ be edges in the disjoint union of $A$ and $B$. Our required outcome is clear when $a$ and $b$ both lie in one of the graphs so we assume, without loss of generality, that $a$ is an edge of $A$ and $b$ is an edge of $B$.
    Now observe that if there is some $c$ in $C$ so that $a$ and $b$ are the images of $c$ in $A$ and $B$, respectively, then the naturality of $\gamma$ implies that the edges $\gamma_A(a)$ and $\gamma_B(b)$ are related in the span $\cat{source}(A)\hookleftarrow \cat{source}(C)\hookrightarrow\cat{source}(B)$ by $\gamma_C(c)$. Now assume that $a$ and $b$ are not related by the span. Then they do not share a source in~$G$ and do not share a source in $\cat{source}(A)+\cat{source}(B)$. This completes our proof.
\end{proof}

\body{
    We call the lasso $(\cat{source},\sigma)$ defined in the above the \define{source lasso}.
}

\begin{proposition}\label{prop:target}
    There is a component-collapsing lasso $(\cat{target},\tau)$ so that for each graph $G$, the map $\tau_G$ identifies exactly the edges having $u$ as a target, for each vertex $u$ in $G$.
\end{proposition}
\begin{proof}
    The proof is exactly analogous to the proof of Proposition~\ref{prop:source}.
\end{proof}
\body{
    We call the lasso $(\cat{target},\tau)$ defined in the above the \define{target lasso}.
}

\begin{proposition}\label{prop:gather}
    There is a component-collapsing lasso $(\cat{gather},\gamma)$ so that for each graph $G$, the map $\gamma_G$ identifies exactly the edges of $G$ which are loops in $c$, for each connected component $c$ of $G$.
\end{proposition}
\begin{proof}
    For each reflexive graph $G$, let $\cat{gather}(G)$ be the graph obtained from $\cat{cc}(G)$ by identifying the loops of $G$ in $c$, for each connected component $c$ of $G$. It is easy to see that we can extend $\cat{gather}$ to a functor $\cat{gather}\colon\cat{RGrph}\to\cat{RGrph}$ and that the maps $\gamma_G$ form a natural transformation $\gamma\colon\id_\cat{RGrph}\to\cat{gather}$.

    Now consider an arbitrary monic span $A\hookleftarrow C\hookrightarrow B$ in $\cat{RGrph}$ and denote by $G$ its pushout. We are required to show that quotienting edges by the relation of being loops in the same connected component commutes with monic pushouts. This is clear, since $\gamma_G$ preserves connected components for each $G$ and also preserves the property of being a loop.
\end{proof}
\body{
    We call the lasso $(\cat{gather},\gamma)$ defined in the above the \define{gathering lasso}. Along with the source and target lassos, we shall see that we have finally exhausted the space of lassos on $\cat{RGrph}$.
}

\begin{lemma}\label{lem:not_just_para}
    Let $(\La,\eta)$ be a component-collapsing lasso. Suppose that some component of $\eta$ identifies a pair of parallel non-loop edges. Then some component of $\eta$ must identify a pair of non-parallel edges.
\end{lemma}
\begin{proof}
    Let $G$ be a graph so that $\eta_G$ identifies a pair $e,e'$ of parallel edges in $G$. Then analogously to in the proof of Lemma~\ref{lemma:parallel_edge_particle}, we get that all parallel classes of edges in all graphs are identified to single edges by $\eta$. Suppose for a contradiction that no component of $\eta$ identifies edges which are not in parallel. Consider the reflexive graph \suprisereverseedge consisting of a pair of vertices, spanning which is a pair of parallel edges and a single other non-loop edge in the opposing direction. Then we may express $\suprisereverseedge$ as a monic pushout over graphs having no parallel edges as in the following illustration.
    \begin{center}
        \includegraphics[width=.8\textwidth]{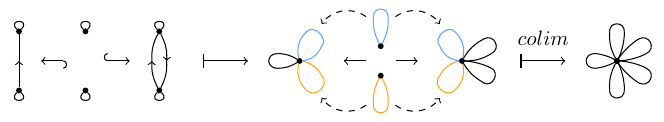}
    \end{center}
    However we observe that the colimit of the image of this diagram under $\La$ does not have the original parallel edges identified. Hence $\La$ does not preserve monic pushouts, a contradiction.
\end{proof}

\begin{lemma}\label{lem:hom_tool}
    Let $G$ be a reflexive graph and $e$ and $e'$ distinct edges in $G$ which are not parallel but lie in the same connected component. Let~$\fulltwo$ denote the reflexive graph having two vertices and an edge in each direction between them. Let $\edgewithloops$ denote the reflexive graph having two vertices and just one edge between them. Then there is an epimorphism $f$ from $G$ to either $\fulltwo$ or $\edgewithloops$, so that one of the following holds:
    \begin{enumerate}
        \item[(1)] the edges $f(e)$ and $f(e')$ are distinct loops; or
        \item[(2)] the edges $f(e)$ and $f(e')$ are a loop and a non-loop, respectively, or vice versa.
    \end{enumerate}
\end{lemma}
\begin{proof}
    First suppose that $e$ and $e'$ are both loops and denote by $u$ and $u'$ the vertices which they are attached to. Let $T$ be a spanning tree in $G$ and let $d$ be an edge that separates $u$ and $u'$ in $T$. Now identify all of the vertices in each of the two components of $T-d$. The resulting graph has two vertices and from this we can easily obtain our function~$f$ satisfying (1).
    In the case where one of $e$ is a loop and the other is a non-loop, we perform the same operation where $d$ is the non-loop edge and $T$ is an arbitrary spanning tree containing $d$. This gives us outcome (2). Finally, in the case where $e$ and $e'$ are both non-loops, we choose one of these edges arbitrarily to be $d$ and the spanning tree $T$ so that it contains this chosen edge. In this case, we also have outcome (2).
\end{proof}

\begin{lemma}
    Let $(\La,\eta)$ be a component-collapsing lasso. If $(\La,\eta)$ is not edge-trivial, then either it is:
    \begin{enumerate}
        \item the lasso $(\cat{cc},\pi)\circ(\cat{deloop},\delta)$;
        \item one of the lassos $(\cat{source},\sigma)$, $(\cat{target},\tau)$, $(\cat{gather},\gamma)$; or
        \item the lasso $(\cat{deloop},\delta)\circ(\cat{cc},\pi)$, which is the terminal lasso on $\cat{RGrph}$.
    \end{enumerate}
\end{lemma}
\begin{proof}
    We first observe that, by Lemma~\ref{lemma:all_by_canon}, the components of $\eta$ identify all classes of parallel loops. If this is the only edge-identification which is performed, then $(\La,\eta)$ is the lasso $(\cat{cc},\pi)\circ(\cat{deloop},\delta)$.
    
    Now suppose that, in addition to this identification, there is some graph $G$ so that $\eta_G$ identifies a pair of edges $e,e'$ which are not a pair of parallel loops. By Lemma~\ref{lem:not_just_para}, we may assume that $e$ and $e'$ are not parallel.
    
    Let $f$ be the graph homomorphism given by Lemma~\ref{lem:hom_tool} and split into cases based on the outcomes (1) and (2).

    \textbf{Case 1.} First suppose that the codomain of $f$ is the graph $\edgewithloops$. That is, $f$ is a homomorphism from $G$ to $\edgewithloops$ so that $e$ and $e'$ are mapped to distinct loops. Then since~$\eta_G$ identifies $e$ and $e'$ and by the naturality of $\eta$, we must have that $\eta_\edgewithloops$ identifies the two loops of $\edgewithloops$. Now given any reflexive graph $H$, one may inject $\edgewithloops$ into $H$ to show that any two loops either end of an edge are identified under $\eta_H$. By extension, any two loops in the same connected component in $H$ are identified by $\eta_H$.
    We claim that the same outcome holds in the case that the codomain of $f$ is the graph $\fulltwo$. Indeed, since $\fulltwo$ can be expressed as the monic pushout of two copies of $\edgewithloops$ over the discrete reflexive graph $\twodiscrete$ on two vertices, the preservation of monic pushouts implies that $\eta_\edgewithloops$ identifies the loops of $\edgewithloops$.
    We conclude this case by noting the following: if the identification of loops in the same connected component is the only additional identification of edges on top of identifying parallel loops, then we get that $(\La,\eta)$ is the gathering lasso $(\cat{gather},\gamma)$.

    \textbf{Case 2.} First suppose that the codomain of $f$ is the graph $\edgewithloops$. So $f$ is a homomorphism from $G$ to $\edgewithloops$ so that $e$ and $e'$ are mapped to one loop and one non-loop. Then since $\eta_G$ identifies $e$ and $e'$ and by the naturality of $\eta$, we must have that $\eta_\edgewithloops$ identifies a loop with a non-loop from $\edgewithloops$. There are two subcases. Either the non-loop is identified with the loop at its source or identified with the loop at its target. Now given any reflexive graph $H$, one may inject $\edgewithloops$ into $H$ to show that any edges sharing a source (resp. a target) are identified under $\eta_H$.
    In the case where the codomain of $f$ is $\fulltwo$, we proceed analogously to in Case 1 and obtain the same conclusion. That is, the components of $\eta$ identify all edges sharing a source or identify all edges sharing a target. If these are the only identifications on top of identifying parallel loops, then $(\La,\eta)$ is either $(\cat{source},\sigma)$ or $(\cat{target},\tau)$.

    The conclusion of the proof comes from the following observation: out of the three relations we obtain as outcomes of Cases 1 and 2, if $\gamma$ identifies via more than one of them, then $(\La,\eta)$ is the terminal lasso. That is, the lasso which identifies all edges belonging to the same connected component.
\end{proof}

\body{
The above lemma completes our analysis of the category of lassos on the category $\cat{RGrph}$ of reflexive graphs, giving us Figure~\ref{fig:rgrph_lassos} from Section~\ref{sec:cat_of_lassos}.
}

\section{Proof of Lemma~\ref{lemma:definition-of-Q}}
\label{appendix:laxity}

\body{
We begin by first recalling background results due to~\cite{structured-decompositions}.
}

\begin{lemma}[Lemma 2.5.13 in~\cite{structured-decompositions}]\label{lemma: monic diagram presheaf}
Suppose that $\mathcal{I}$ is a class of all (small) monic diagrams and $\cat{C}$ is a category with pullbacks and pullback-stable colimits of type $\mathcal{I}$. By choosing representatives for every pullback, one obtains a well-defined presheaf 
\begin{equation*}
    \mathsf{MDgm} \colon \cat{C}^{\op} \to \cat{Set},
\end{equation*}
where for $X \in \cat{C}$, $\cat{MDgm}(X)$ is the set of diagrams $d \colon I \to \cat{C}$ in $\mathcal{I}$ equipped with a colimit cocone $\sigma \colon d \Rightarrow \Delta(X)$. Given a map $f \colon X \to Y$ in $\cat{C}$, $\cat{MDgm}(f)$ is the function that takes a colimit cocone $\sigma$ over $Y$ to its pullback $f^*(\sigma)$ over $X$.
\end{lemma}

\body{
It turns out we can promote the functor $\mathsf{MDgm}$ of the above lemma to a functor into $\cat{Cat}$ (the category of small categories) as shown in the following corollary. The proof of Corollary~\ref{cor:submon_extension} is straightforward and mechanical, but, since it is more categorically involved than the rest of the document, we defer it to Appendix~\ref{appendix:laxity} in order to not distract from the main points.
}

\begin{corollary}[Corollary~\ref{cor:submon_extension} restated]
    The functor $\mathsf{MDgm}:\cat{C}^{\op} \to\mathbf{Set}$ in Lemma~\ref{lemma: monic diagram presheaf} extends to a functor $\Ma:\cat{C}^{\op}\to \mathbf{Cat}$.
\end{corollary}
\begin{proof}[Proof of Corollary~\ref{cor:submon_extension}]
    For each $X\in\cat{C}$, define $\Ma(X)$ be the category of structured decompositions which yield $X$ as a colimit. For each $f\colon X\to Y$ in $\cat{C}$, we are required to define a functor $\Ma(f)\colon\Ma(X)\to\Ma(Y)$. On objects, the action of this functor is analogous to $\mathsf{MDgm}(f)$. That is, \[
\Ma(f)\colon(d_Y\in\Ma(Y))\mapsto(d_X\in\Ma(X)).
    \]
where $d_X$ is obtained from $d_Y$ by pointwise pullbacks along $f$. Finally, given a map $(F, 
\eta) \colon d_Y\to d_Y'$ between structured decompositions which have colimit $Y$, we define the action of $\Ma(f)$ on $(F,\eta)$ to be given by the collection of unique pullback arrows $\varphi_i$ making the following commute.
\[\begin{tikzcd}
	& {d_Y(Fi)} && {d'_Y(Fi)} \\
	{d_X(Fi)} && {d'_Y(Fi)} && Y \\
	& X && X
	\arrow["{\eta_i}", from=1-2, to=1-4]
	\arrow[from=1-4, to=2-5]
	\arrow[from=2-1, to=1-2]
	\arrow["{\exists ! \varphi_i}"{description}, dashed, from=2-1, to=2-3]
	\arrow[from=2-1, to=3-2]
	\arrow[from=2-3, to=1-4]
	\arrow["\lrcorner"{anchor=center, pos=0.125, rotate=45}, draw=none, from=2-3, to=2-5]
	\arrow[from=2-3, to=3-4]
	\arrow["{\mathsf{id}_X}", from=3-2, to=3-4]
	\arrow["f"', from=3-4, to=2-5]
\end{tikzcd}\]
Since $d_X$ has the same shape as $d_Y$, and $d_X'$ the same shape as $d_Y'$, this properly defines a morphism of decompositions of the form $(F, \varphi)$ where the $i$-th component to $\varphi$ is $\varphi_i$ as in the diagram above. Thus this defines the action of $\Ma$ on morphisms since one then has $\Ma(f)(F, \eta) := (F, \varphi)$. This assignment is functorial since identities are preserved and, by the universal property of the morphisms $\varphi_i$, composition is respected.
\end{proof}

\begin{lemma}[Lemma~\ref{lemma:definition-of-Q} restated]
The functions in Lemma~\ref{def_Q} assemble into a \textit{lax} natural transformation
\[\begin{tikzcd}
	{\cat{C}^{\op}} && {\cat{Cat}}
	\arrow[""{name=0, anchor=center, inner sep=0}, "\Ma", curve={height=-12pt}, from=1-1, to=1-3]
	\arrow[""{name=1, anchor=center, inner sep=0}, "{\Ma \circ \La^{\op}}"', curve={height=12pt}, from=1-1, to=1-3]
	\arrow["\Qa", shorten <=3pt, shorten >=3pt, Rightarrow, from=0, to=1]
\end{tikzcd}\]
\end{lemma}
\begin{proof}[Proof of Lemma~\ref{lemma:definition-of-Q}]
Recall that the function \[\Qa_X \colon \Ma X \to \Ma \La X,\] acts by sending a diagram $d\colon\cat{J}\to\cat{C}$ with colimit colone $(\Lambda,X)$ to the diagram of images under $\eta_X \Lambda$.
By Lemma~\ref{def_Q}, the action of $\Qa_X \colon \Ma X \to \Ma \La X$ on any object $d \in \Ma X$ is indeed a diagram whose colimit is $\La X$.

Having given the action of $\Qa_X$ on objects, we will turn it into a functor by defining its action on morphisms. This will follow immediately by the universal property of images. To that end, consider the following morphism of diagrams $d \to d'$ in $\Ma X$: \[(F, \alpha) \colon \bigl ( \cat{J} \xrightarrow{d} \cat{C} \bigr) \to \bigl ( \cat{J'} \xrightarrow{d'} \cat{C} \bigr).\]
For any $i \in \cat{J}$, this yields the following diagram where one deduces the existence of the monomorphism $q_i$ by the universal property of images (and that it is a monomorphism by Paragraph~\ref{para:mono-triangle}).  
\begin{equation}\label{diagram:laxity-functor}
\begin{tikzcd}
	& {d(i)} \\
	{d'(Fi)} &&&& X \\
	& {\La d(i)} & {\Qa(d)(i)} \\
	{\La d'(Fi)} &&&& {\La X} \\
	&& {\Qa(d')(Fi)}
	\arrow["{\alpha_i}"{description}, from=1-2, to=2-1]
	\arrow["{\Lambda_{i}}"{description}, from=1-2, to=2-5]
	\arrow["{\eta_{i}}"{description, pos=0.3}, from=1-2, to=3-2]
	\arrow["{\Lambda'_{Fi}}"{description, pos=0.3}, hook, from=2-1, to=2-5]
	\arrow["{\eta_{Fi}}"{description, pos=0.3}, two heads, from=2-1, to=4-1]
	\arrow["{\eta_X}"{description, pos=0.3}, two heads, from=2-5, to=4-5]
	\arrow[two heads, from=3-2, to=3-3]
	\arrow["{\La \alpha_i}"', from=3-2, to=4-1]
	\arrow["{\La \Lambda_{i}}", curve={height=-37pt}, from=3-2, to=4-5]
	\arrow[hook, from=3-3, to=4-5]
	\arrow["{\exists ! q_i}"{description, pos=0.3}, color={rgb,255:red,251;green,55;blue,228}, hook, from=3-3, to=5-3]
	\arrow["{\La \Lambda'_{Fi}}"{description, pos=0.3}, from=4-1, to=4-5]
	\arrow["{\beta_1}"', two heads, from=4-1, to=5-3]
	\arrow["{\beta_2}"', hook, from=5-3, to=4-5]
\end{tikzcd}
\end{equation}
The morphisms $q_i$ assemble into a natural transformation $\beta$ which we will use to define the morphism \[\Qa_X(F, \alpha) \colon \Qa_X(d) \to \Qa_x(d')\] as $\Qa_X(F, \alpha) := (F, \beta)$. Having defined the action of $\Qa_X$ on objects and morphisms, it is easy to check functoriality: preservation of identities is immediate and preservation of composition follows by Lemma~\ref{lemma:images_map}. 

All that remains is to show that $\Qa$ is a lax natural transformation. This amounts to exhibiting a natural transformation $\Qa(f)$ as in the following diagram for every morphism $f \colon X \to Y$ in~$\cat{C}$. 
\[\begin{tikzcd}
	{\Ma X} && {\Ma \La X} \\
	{\Ma Y} && {\Ma \La Y}
	\arrow["{\Qa_X}", from=1-1, to=1-3]
	\arrow["{\Qa_f}"{description}, shorten <=5pt, shorten >=5pt, Rightarrow, from=1-1, to=2-3]
	\arrow["{\Ma f}", from=2-1, to=1-1]
	\arrow["{\Qa_Y}"', from=2-1, to=2-3]
	\arrow["{\Ma \La f}"', from=2-3, to=1-3]
\end{tikzcd}\]
Each component of $\Qa(f)$ will be a morphism of diagrams in $\Ma \La X$. Since neither $\Ma$ nor $\Qa$ change the shapes of any diagram they are applied to, $\Qa(f)$ will be a morphism of the following form (where $\nu$ will be defined below) \begin{equation}\label{eqn:component-of-laxator}
    (\id, \nu)_d \colon (\Qa_X \circ \Ma(f))(d)(i) \to ((\Ma\La)(f) \circ \Qa_Y)(d)(i)
\end{equation}  
Towards establishing the definition of $\nu$, consider -- for $i$ any object of $\domain(d)$ -- the following diagram. 
\begin{equation}\label{diagram:lax-natural-1}
\begin{tikzcd}[row sep=small]
	& {\imageObj{\eta_X \ell_i}} \\
	{\La X} && X & {X \times_Y d(i)} \\
	& {(\La X) \times_{\La Y} (\imageObj{\eta_Y \Lambda_i})} \\
	{\La Y} && Y & {d(i)} \\
	& {\imageObj{\eta_Y \Lambda_i}}
	\arrow[hook, from=1-2, to=2-1]
	\arrow["{\exists ! \nu_i}"{description, pos=0.3}, color={rgb,255:red,214;green,92;blue,92}, dashed, hook', from=1-2, to=3-2]
	\arrow["{\exists ! u_i}"{pos=0.6}, color={rgb,255:red,92;green,92;blue,214}, curve={height=-18pt}, dashed, from=1-2, to=5-2]
	\arrow["{\La f}"', from=2-1, to=4-1]
	\arrow["{\eta_X}"{pos=0.3}, two heads, from=2-3, to=2-1]
	\arrow["f"', from=2-3, to=4-3]
	\arrow[two heads, from=2-4, to=1-2]
	\arrow["{\ell_i}", hook, from=2-4, to=2-3]
	\arrow[from=2-4, to=4-4]
	\arrow[hook', from=3-2, to=2-1]
	\arrow[from=3-2, to=5-2]
	\arrow["{\eta_Y}"'{pos=0.3}, two heads, from=4-3, to=4-1]
	\arrow["{\Lambda_i}"', hook', from=4-4, to=4-3]
	\arrow[two heads, from=4-4, to=5-2]
	\arrow[hook, from=5-2, to=4-1]
\end{tikzcd}
\end{equation}
The blue morphism $u_i$ in the above diagram is deduced by the universal property of images. Using $u_i$ and the universal property of pullbacks, one obtains the unique arrow $\nu_i$ shown in red in Diagram~\ref{diagram:lax-natural-1}. Finally, observing that 
\begin{align*}
    (\Qa_X \circ \Ma(f))(d)(i) & = \imageObj(\eta_X \ell_i) &\text{and } \\
    ((\Ma\La)(f) \circ \Qa_Y)(d)(i) &= (\La X) \times_{\La Y} \imageObj{\eta_Y \Lambda_i},
\end{align*}
one finds that the natural transformation $\nu$ will have as its $i$-component the unique pullback arrow $\nu_i$ of Diagram~\ref{diagram:lax-natural-1}. Thus, as desired, we have now defined, for any morphism $f \colon X \to Y$, the natural transformation $\Qa(f)$ whose component at any object $d \in \Ma Y$ is given by \[\Qa(f)_d := (\id, \nu)_d \quad \text{ (as in Equation~\ref{eqn:component-of-laxator}}).\] 

All that remains to be shown is that $\Qa(f)$ is natural. To that end, take any morphism of diagrams \[(F, \mu) \colon (\cat{J_1} \xrightarrow{d_1} \cat{C}) \to (\cat{J_2} \xrightarrow{d_2} \cat{C})\] in $\Ma Y$. By two applications of the exactly the same reasoning as in Diagram~\ref{diagram:lax-natural-1}, on obtains the red and blue portions of the following diagram (the red portion corresponds to $d_2$ and the blue portion corresponds to $d_1$).
\begin{equation}\label{diagram:lax-natural-2}
\adjustbox{scale=1.5, max width=\textwidth}{
\begin{tikzcd}
	&& {\imageObj{\eta_X \ell_{1,i}}} \\
	&&&& {X \times_Y d_1(i) = \bigl ( X \times_Y d_2(Fi) \bigr ) \times_{d_2(Fi)} d_1(i) } \\
	&& {(\La X) \times_{\La Y} (\imageObj{\eta_Y \Lambda_{1,i}})} \\
	&&&& {d_1(i)} \\
	& {\imageObj{\eta_X \ell_{2,Fi}}} & {\imageObj{\eta_Y \Lambda_{2,Fi}}} \\
	{\La X} && X & {X \times_Y d_2(Fi)} \\
	& {(\La X) \times_{\La Y} (\imageObj{\eta_Y \Lambda_{2,Fi}})} \\
	{\La Y} && Y & {d_2(Fi)} \\
	& {\imageObj{\eta_Y \Lambda_{2,Fi}}}
	\arrow["{\exists ! \nu_{1,i}}"{description, pos=0.3}, color={rgb,255:red,92;green,92;blue,214}, dashed, hook', from=1-3, to=3-3]
	\arrow["{\exists ! q_1}"{description}, color={rgb,255:red,253;green,53;blue,206}, from=1-3, to=5-2]
	\arrow["{\exists ! u_{1,i}}"{pos=0.6}, color={rgb,255:red,92;green,92;blue,214}, curve={height=-18pt}, dashed, from=1-3, to=5-3]
	\arrow[draw={rgb,255:red,92;green,92;blue,214}, curve={height=18pt}, hook, from=1-3, to=6-1]
	\arrow[draw={rgb,255:red,92;green,92;blue,214}, two heads, from=2-5, to=1-3]
	\arrow[draw={rgb,255:red,92;green,92;blue,214}, from=2-5, to=4-5]
	\arrow["{\ell_{1,i}}"{description}, color={rgb,255:red,92;green,92;blue,214}, hook, from=2-5, to=6-3]
	\arrow[dashed, from=2-5, to=6-4]
	\arrow[draw={rgb,255:red,92;green,92;blue,214}, from=3-3, to=5-3]
	\arrow[draw={rgb,255:red,92;green,92;blue,214}, curve={height=18pt}, hook, from=3-3, to=6-1]
	\arrow["{\exists ! q_3}"{description}, color={rgb,255:red,253;green,53;blue,206}, from=3-3, to=7-2]
	\arrow[draw={rgb,255:red,92;green,92;blue,214}, two heads, from=4-5, to=5-3]
	\arrow["{\Lambda_{1,i}}"{description, pos=0.3}, color={rgb,255:red,92;green,92;blue,214}, from=4-5, to=8-3]
	\arrow["{\mu_i}"{description}, color={rgb,255:red,253;green,53;blue,206}, from=4-5, to=8-4]
	\arrow[draw={rgb,255:red,214;green,92;blue,92}, hook, from=5-2, to=6-1]
	\arrow["{\exists ! \nu_{2,Fi}}"{description, pos=0.3}, color={rgb,255:red,214;green,92;blue,92}, dashed, hook', from=5-2, to=7-2]
	\arrow["{\exists ! u\_{2,Fi}}"{pos=0.6}, color={rgb,255:red,214;green,92;blue,92}, curve={height=-18pt}, dashed, from=5-2, to=9-2]
	\arrow[draw={rgb,255:red,92;green,92;blue,214}, hook, from=5-3, to=8-1]
	\arrow["{\exists ! q_2}"{description}, color={rgb,255:red,253;green,53;blue,206}, from=5-3, to=9-2]
	\arrow["{\La f}"', from=6-1, to=8-1]
	\arrow["{\eta_X}"{description, pos=0.3}, two heads, from=6-3, to=6-1]
	\arrow["f"', from=6-3, to=8-3]
	\arrow[draw={rgb,255:red,214;green,92;blue,92}, two heads, from=6-4, to=5-2]
	\arrow["{\ell_{2,Fi}}"{description}, color={rgb,255:red,214;green,92;blue,92}, hook, from=6-4, to=6-3]
	\arrow[draw={rgb,255:red,214;green,92;blue,92}, from=6-4, to=8-4]
	\arrow[draw={rgb,255:red,214;green,92;blue,92}, hook', from=7-2, to=6-1]
	\arrow[draw={rgb,255:red,214;green,92;blue,92}, from=7-2, to=9-2]
	\arrow["{\eta_Y}"{description, pos=0.3}, two heads, from=8-3, to=8-1]
	\arrow["{\Lambda_{2,Fi}}"{description}, color={rgb,255:red,214;green,92;blue,92}, hook', from=8-4, to=8-3]
	\arrow[draw={rgb,255:red,214;green,92;blue,92}, two heads, from=8-4, to=9-2]
	\arrow[draw={rgb,255:red,214;green,92;blue,92}, hook, from=9-2, to=8-1]
\end{tikzcd}
}
\end{equation}
The existence of the pink arrows $q_1$, $q_2$ and $q_3$ shown in Diagram~\ref{diagram:lax-natural-2} are deduced as follows: $q_1$ and $q_2$ follow by the universal property of images while $q_3$ follows by the universal property of pullbacks. The proof of naturality of $\Qa(f)$ is now complete by observing that:
\begin{align*}
     (\Qa_X \circ \Ma(f))(d_1)(i) & = \imageObj(\eta_x \ell_{1,i}) &\text{and } \\
     (\Qa_X \circ \Ma(f))(d_2)(i) & = \imageObj(\eta_x \ell_{2,Fi}) &\text{and } \\
     (\Qa_X \circ \Ma(f))(F, \mu) &= q_1 &\text{and } \\
    ((\Ma\La)(f) \circ \Qa_Y)(d_1)(i) &= (\La X) \times_{\La Y} \imageObj{\eta_Y \Lambda_{1,i}} &\text{and } \\
    ((\Ma\La)(f) \circ \Qa_Y)(d_2)(i) &= (\La X) \times_{\La Y} \imageObj{\eta_Y \Lambda_{2,Fi}} &\text{and } \\
    ((\Ma\La)(f) \circ \Qa_Y)(F, \mu) &= q_3 &\text{and } \\
    \nu_{2, Fi} \circ q_1 &= q_3 \circ \nu_{1,i}.
\end{align*}
\end{proof}

\body{
Consider a strong lasso $(\La, \eta)$ and the functor $\Ma$ of Corollary~\ref{cor:submon_extension} as shown below.

\[\begin{tikzcd}
	{\cat{C}^{\op}} && {\cat{C}^{\op}} && {\cat{Cat}}
	\arrow[""{name=0, anchor=center, inner sep=0}, "{\id^{\op}}", curve={height=-12pt}, from=1-1, to=1-3]
	\arrow[""{name=1, anchor=center, inner sep=0}, "{\La^{\op}}"', curve={height=12pt}, from=1-1, to=1-3]
	\arrow["\Ma", from=1-3, to=1-5]
	\arrow["\eta", shorten <=3pt, shorten >=3pt, Rightarrow, from=0, to=1]
\end{tikzcd}\]
From the above, one defines the natural transformation
\[\begin{tikzcd}
	{\cat{C}^{\op}} &&& {\cat{Cat}}
	\arrow[""{name=0, anchor=center, inner sep=0}, "\Ma", curve={height=-12pt}, from=1-1, to=1-4]
	\arrow[""{name=1, anchor=center, inner sep=0}, "{\Ma \circ \La^{\op}}"', curve={height=12pt}, from=1-1, to=1-4]
	\arrow["{\Ma \circ \eta}"', shorten <=3pt, shorten >=3pt, Rightarrow, from=1, to=0]
\end{tikzcd}\]
via ``whiskering''. That is, where the component of $\Ma \circ \eta$ at any object $x$ is simply the action \[\Ma(X \xrightarrow{\eta_x} \La X) \colon \Ma (\La X) \to \Ma(X)\]
taking diagrams whose colimits are $\La X$ to diagrams whose colimits are $X$. 
As we saw above, the construction of Paragraph~\ref{para:defining-Q} yields a \textit{lax} natural transformation in \textit{the opposite direction} as shown in Lemmas~\ref{def_Q} and~\ref{lemma:definition-of-Q}. The fact that the natural transformation in question is lax hints at deeper categorical structure which, despite being interesting, we leave as future work.
}

\section{Double Category of Lassos}
\label{appendix:double}

\body{
Recall the definition of the category of lassos on $\cat{C}$ from Section~\ref{sec:cat_of_lassos}. We can view maps between lassos as 2-cells as in
\[\begin{tikzcd}
	\bullet && \bullet \\
	\\
	\bullet && \bullet
	\arrow[""{name=0, anchor=center, inner sep=0}, "{(\La,\eta)}", from=1-1, to=1-3]
	\arrow["\id"', from=1-1, to=3-1]
	\arrow["\id", from=1-3, to=3-3]
	\arrow[""{name=1, anchor=center, inner sep=0}, "{(\La',\eta')}"', from=3-1, to=3-3]
	\arrow["f", shorten <=9pt, shorten >=9pt, Rightarrow, from=0, to=1]
\end{tikzcd}\]
}

\body{
Furthermore, give four lassos $(\La,\eta)$,$(\La',\eta')$, $(\Ma,\mu)$ and $(\Ma',\mu')$, along with two maps $ (\La,\eta)\xrightarrow{f}(\La'\eta')$ and $(\Ma,\mu)\xrightarrow{g}(\Ma',\mu')$ in the category of lassos, we can define a kind of `horizontal composition' in the following way. Let $g*f$ be the collection of maps given by the diagonals of the following commutative square.
\[\begin{tikzcd}
	{\Ma\La X} && {\Ma'\La X} \\
	\\
	{\Ma\La' X} && {\Ma'\La' X}
	\arrow["{g_{\La X}}", from=1-1, to=1-3]
	\arrow["{\Ma f_X}"{description}, from=1-1, to=3-1]
	\arrow["{(g*f)_X}"{description}, from=1-1, to=3-3]
	\arrow["{\Ma'f_X}"{description}, from=1-3, to=3-3]
	\arrow["{g_{\La'X}}", from=3-1, to=3-3]
\end{tikzcd}\]

In the above context, we make the following proposition.
}

\begin{proposition}
    Given a category $\cat{C}$, the category of lassos on $\cat{C}$ is a double category with a single object. 
\end{proposition}
\begin{proof}
Given the following lassos and maps between them ($2$-cells) in the category of lassos
\[\begin{tikzcd}
	\bullet && \bullet && \bullet \\
	\bullet && \bullet && \bullet \\
	\bullet && \bullet && \bullet
	\arrow[""{name=0, anchor=center, inner sep=0}, "{(\La,\eta)}"{description}, from=1-1, to=1-3]
	\arrow["\id"{description}, from=1-1, to=2-1]
	\arrow[""{name=1, anchor=center, inner sep=0}, "{(\Ma,\mu)}"{description}, from=1-3, to=1-5]
	\arrow["\id"{description}, from=1-3, to=2-3]
	\arrow["\id"{description}, from=1-5, to=2-5]
	\arrow[""{name=2, anchor=center, inner sep=0}, "{(\La',\eta')}"{description}, from=2-1, to=2-3]
	\arrow["\id"{description}, from=2-1, to=3-1]
	\arrow[""{name=3, anchor=center, inner sep=0}, "{(\Ma',\mu')}"{description}, from=2-3, to=2-5]
	\arrow["\id"{description}, from=2-3, to=3-3]
	\arrow["\id"{description}, from=2-5, to=3-5]
	\arrow[""{name=4, anchor=center, inner sep=0}, "{(\La',\eta'')}"{description}, from=3-1, to=3-3]
	\arrow[""{name=5, anchor=center, inner sep=0}, "{(\Ma'',\mu'')}"{description}, from=3-3, to=3-5]
	\arrow["f", shorten <=4pt, shorten >=4pt, Rightarrow, from=0, to=2]
	\arrow["g", shorten <=4pt, shorten >=4pt, Rightarrow, from=1, to=3]
	\arrow["{f'}", shorten <=4pt, shorten >=4pt, Rightarrow, from=2, to=4]
	\arrow["{g'}", shorten <=4pt, shorten >=4pt, Rightarrow, from=3, to=5]
\end{tikzcd}\]
we are required to show that vertical and horizontal composition of the $2$-cells is associative. Consider the following commutative diagram.
\[\begin{tikzcd}
	{\Ma\La X} && {\Ma'\La X} && {\Ma''\La X} \\
	\\
	{\Ma\La' X} && {\Ma'\La' X} && {\Ma''\La' X} \\
	\\
	{\Ma\La''X} && {\Ma'\La''X} && {\Ma''\La''X}
	\arrow["{g_{\La X}}", from=1-1, to=1-3]
	\arrow["{\Ma f_X}"{description}, from=1-1, to=3-1]
	\arrow["{(g*f)_X}"{description}, from=1-1, to=3-3]
	\arrow["{\Ma(f'\circ f)_X}"{description, pos=0.6}, curve={height=30pt}, from=1-1, to=5-1]
	\arrow["{g'_{\La X}}", from=1-3, to=1-5]
	\arrow["{\Ma'f_X}"{description}, from=1-3, to=3-3]
	\arrow["{\Ma''f_X}"{description}, from=1-5, to=3-5]
	\arrow["{g_{\La'X}}", from=3-1, to=3-3]
	\arrow["{\Ma f'_X}"{description}, from=3-1, to=5-1]
	\arrow["{g'_{\La'X}}", from=3-3, to=3-5]
	\arrow["{\Ma' f'_X}"{description}, from=3-3, to=5-3]
	\arrow["{(g'*f')_X}"{description}, from=3-3, to=5-5]
	\arrow["{\Ma''f'_X}"{description}, from=3-5, to=5-5]
	\arrow["{g_{\La''X}}", from=5-1, to=5-3]
	\arrow["{(g'\circ g)_{\La''X}}"{description}, curve={height=24pt}, from=5-1, to=5-5]
	\arrow["{g'_{\La''X}}", from=5-3, to=5-5]
\end{tikzcd}\]
Composing the two long horizontal and vertical arrows yields $((g'\circ g)*(f'\circ f))_X$ and hence
\(
(g'\circ g)*(f'\circ f) = (g*f)\circ (g'*f')
\)
as required.
\end{proof}

\bibliographystyle{plain}
\bibliography{wliterature}

\end{document}